\documentclass[a4paper,12pt]{article}
\usepackage{amsmath,amsthm}
\usepackage[margin=1.2in]{geometry}
\def\={\!=\!}

\usepackage{hyperref}
\usepackage{listings}
\usepackage{xcolor}
\newtheorem{thm}{Theorem}
\newtheorem{lem}{Lemma}
\newtheorem{cor}{Corollary}
\newtheorem{prop}{Proposition}
\newtheorem{definition}{Definition}
\newtheorem{que}{Question}

\def\={\!=\!}

\def\|{{\Vert}}

\newcommand{\noi}{\noindent}

\newcount\rem
\def\Remark{\skip\noi{{\bf{Remark \number\rem.}}} \advance\rem by 1}

\title{
A Continuous Paradoxical Colouring Rule Using Group Action}
\author{Tu\u{g}kan Batu, Robert Simon, Grzegorz Tomkowicz}

\begin{document}
\begin{titlepage}
\maketitle
\thispagestyle{empty}

\begin{abstract}
  Given a probability space $(X, {\cal B}, m)$, measure preserving
  transformations $g_1, \dots , g_k$ of $X$, and a colour set $C$, a colouring
  rule is a way to colour the space with $C$ such that the colours allowed for
  a point $x$ are determined by that point's location and the colours of the
  finitely $g_1 (x), \dots , g_k(x)$ with $g_i(x) \not= x$ for all $i$ and
  almost all $x$.  We represent a colouring rule as a correspondence $F$
  defined on $X\times C^k$ with values in $C$. A function $f: X\rightarrow C$
  satisfies the rule at $x$ if $f(x) \in F( x, f(g_1 x), \dots , f(g_k x))$.  A
  colouring rule is paradoxical if it can be satisfied in some way almost
  everywhere with respect to $m$, but not in {\bf any} way that is measurable
  with respect to a finitely additive measure that extends the probability
  measure $m$ and for which the finitely many transformations
  $g_1, \dots , g_k$ remain measure preserving. We show that a colouring rule
  can be paradoxical when the $g_1, \dots, g_k$ are members of a group $G$, the
  probability space $X$ and the colour set $C$ are compact sets, $C$ is convex
  and finite dimensional, and the colouring rule says if $c: X\rightarrow C$ is
  the colouring function then the colour $c(x)$ must lie ($m$ a.e.)  in
  $F(x, c(g_1(x) ), \dots , c(g_k(x)))$ for a non-empty upper-semi-continuous
  convex-valued correspondence $F$ defined on $X\times C^k$.  We show that any
  colouring that approximates the correspondence by $\epsilon$ for small enough
  positive $\epsilon$ cannot be measurable in the same finitely additive way.
  Furthermore any function satisfying the colouring rule illustrates a paradox
  through finitely many measure preserving shifts defining injective maps from
  the whole space to subsets of measure summing up to less than one.
  \end{abstract}
\vskip.1cm

\noi Tu\u{g}kan Batu and Robert Simon \newline London School of
Economics,  Department of Mathematics\newline Houghton Street,
London WC2A 2AE\newline e-mail: \{t.batu, r.s.simon\}@lse.ac.uk\newline
\vskip.1cm
 
\noi Grzegorz Tomkowicz\\
Centrum Edukacji $G^2$,
ul.Moniuszki 9,
41-902 Bytom,
Poland\\
e-mail: gtomko@vp.pl

\end{titlepage}



\section{Introduction}

A common belief is that measure theoretic paradoxes, like the Banach-Tarski
paradox, are not relevant to real events. The reasoning is that such paradoxes
require the Axiom of Choice (AC), and therefore the exhibition of paradoxical
behaviour would negate the fact that we can remain logically consistent after
rejecting AC (Note however that there is a class of theorems called shadows of
AC which proofs use AC and exist even after rejecting AC; see~\cite{MT} for the
details.)

There are two parts to showing that a decomposition is paradoxical, a part
showing that it cannot be done in a measurable way and a part showing its
existence. It is the second part that requires some variation of AC. We don't
employ AC for the first part.

In~\cite{ST1}, we considered colouring rules such that the allowed colours for
a point are determined by location and the colours of finitely many of the
point's neighbours in a graph.  The adjacency relation is defined by measure
preserving transformations and the finitely many transformations of a point $x$
are called its {\em descendants}. A colouring rule is paradoxical if it can be
satisfied in some way almost everywhere, but not in any way that is measurable
with respect to a finitely additive measure for which the transformations
defining the descendants remain measure preserving. We demonstrated several
paradoxical colouring rules and proved that if the measure preserving
transformations belong to a group and there are finitely many colour classes
then any colouring of a paradoxical colouring rule has colour classes that
jointly with the measure preserving transformations and the Borel sets define a
{\em measurably $G$-paradoxical decomposition } of the probability space, by
which we mean the existence of two measurable sets of different measures that
are $G$-equidecomposable (see~\cite{ST1}, Thm. 1).

In the conclusion of~\cite{ST1}, we asked whether a colouring rule could be
paradoxical if the colour classes belonged to a finite dimensional convex set
and the colouring rule was defined by an upper-semi-continuous convex-valued
non-empty correspondence. We call such colouring rules {\em probabilistic}
colouring rules. In \cite{ST2}, we demonstrated a paradoxical probabilistic
colouring rule where some of the measure preserving transformations were
non-invertible, hence do not belong to a group. In the present paper, we give
an example where they do belong to a group. Invertibility of the
transformations defining the descendants presents special challenges.  To show
that a colouring rule is paradoxical it is advantageous to show that
satisfaction of the rule is done with extremal points (of colours) almost
everywhere. It is easier to do this if the transformations are
uncountable-to-one.  We do it below with transformations that are one-to-one.
    
One  interpretation of the upper-semi-continuous convex valued
correspondence is that the choosing of colours is according to a maximisation
or minimisation of a continuous and affine evaluation of options, with
indifference between two options implying indifference between all of their
convex combinations.  This is such an example.
          
Throughout this paper, by a {\em proper} finitely additive extension we
mean a finitely additive measure that extends the Borel measure and is
invariant with respect to the measure preserving actions used.

We use an inclusive concept of approximation.  A solution is $\epsilon$-stable
if the expected gains from deviation, evaluated at each point individually and
then integrated over the space, are no more than $\epsilon$.  We show that
every $\epsilon$-stable solution for small enough $\epsilon$ cannot be
measurable with respect to any proper finitely additive extension.

We were inspired optimistically by an example from~\cite{ST1}.  A brief description of
that example follows.

Consider $G= {\mathbf F}_2$ (the group generated freely by two generators) and
the space~$X=\{ -1, +1\} ^G$ acted on by $G$ in the canonical way (through
shifting). Let $g_1, g_2$ be the free generators of $G$.  Where the $e$
coordinate of an $x\in X$ lies in~$\{ -1, +1\}$ determines whether an arrow
from $x$ should be directed toward the choice of either $g_1x$ or $g_2x$ (if
$x^e = +1$) or rather toward the choice of either $g_1^{-1}x$ or~$g_2^{-1}x$
(if~$x^e= -1$).  If a point in the space $X$ has two or more arrows directed to
it, it is {\em congested}; if not, it is {\em uncongested}. The rule is to
direct an arrow, if possible, toward a point that is uncongested (if not
possible or possible in both directions, the rule allows the arrow to be
directed in either direction). We showed that if almost all points follow this
rule, the set of congested points is a subset of measure zero. Let the degree
of $x$ be the total number of potential inward arrows toward $x$, as determined
by the choice of~$-1$ or $+1$ for the $e$ coordinate of its four neighbours by
the four directions $g_1, g_2, g_1^{-1}, g_2^{-1}$.  The argument, that the set
of congested points is of measure zero when the rule is satisfied, was
combinatorial. The structure supporting a congested point requires a continuous
repetition of three or four degree points in an infinite chain of
vertices. Because the average degree of a vertex is two, such a structure is
restricted to a subset of measure zero.  This implied that the rule was
paradoxical, because in at least $\frac 1 {16}$ of the space there was no
possibility of any arrow moving inward (and that probability is slightly higher
than $\frac 1 {16}$, due to some configurations of two or more points with no
possibility of arrows coming in toward that configuration). From every point in
the space there was one arrow going outwards, but going inwards on the average
there could be at most $\frac {15}{16}$ (assuming a measurable structure). This
approach followed an inspiration from combinatorics, inclusion-exclusion. It is
an example of a probability space $(X, {\cal B}, \mu)$ partitioned into $n$
parts $A_1, \dots , A_n$ such that after moving these parts by invertible
measurable preserving transformations $\sigma_1, \dots , \sigma_n$, we have
$\mu (X\backslash (\cup_{i=1}^n \sigma_i A_i) ) >0$ and for every $i\not= j$
$\mu (\sigma_i A_i \cap \sigma_j A_j)=0$.  By inclusion-exclusion this is a
contradiction to measurability.

One could try to repeat the argument with a weighted choice between the two
options (meaning some distribution $(p_1, p_2)$ with $0\leq p_1 \leq 1$ and
$p_1 + p_2=1$), and devise some definition of what is means to be congested and
how congested. With a rule that requires pointing the arrow to the less
congested option, one could hope to reproduce a situation where the average
weight of received arrows would be strictly less than $1$. With discrete
choices there is no option between $1$ arrow coming in and $2$ arrows coming
in.  With weighted choices, however, using the critical weight of
$1+ \frac 1 {16}$ to defined crowded, there are too many ways to distribute
weights so that crowdedness appears throughout typical infinite chains of
connected vertices. We kept the same idea that the differences between the
degrees of vertices is a useful stochastic process, but we had to look for more
sophisticated ways to determine how weights should be distributed.

Our idea was that there should be two kinds of crowdedness, a crowdedness at a
point receiving weights, what we call {\em passive pain}, and another kind of
crowdedness at the location from where weights come, what we call {\em active
  pain}.  The idea was that if $x$ could direct some weight to $y$ it is not
the crowdedness at~$y$, the passive pain, that counts at $x$ but the total
weight sent from~$x$ to $y$ times that passive pain. This has the effect of
distributing the weights more evenly. It also relates the pain levels to
entropy type inequalities, critical to the proof.  Note that the colouring rule
for some~$x$ cannot be determined in any way by the colour of $x$, however we
incorporate into the colour of $y$ a variable that reflects the weight coming
from $x$.  Rather than a combinatorial argument as before, we show that pain,
both active and passive, almost everywhere has to increase on the average along
an infinite chain.  As there is a finite upper limit to the level of both kinds
of pain, that would mean that pain can be sustained only in a set of measure
zero.
                 
We discovered that two choices for distributing weights was not
enough. Though the group generated freely by two independent elements has
arbitrarily many independent choices, to keep the structure simple we equated
generators with choices.  With~$n$ generators, as before, $-1$ for the $e$
coordinate of $x$ means that weight can be sent only in the negative directions
(to $g_1^{-1} x,\dots , g_n ^{-1} x$) and~$+1$ means that weight can be sent
only in the positive directions (to $g_1^{+1} x,\dots , g_n ^{+1} x$).  Define
the degree of a point $y$ in $X$ to be the number of directions from which
weight can be be sent toward $y$.  With~$n$ generators, the degree ranges from
$0$ to $2n$ with an average of $n$. The distribution of degrees is determined
by the binomial expansion. 

No matter how many generators were used, there is always a possibility for both
passive and active pain to decrease. As a general rule, smaller degrees in a
chain means an increase in pain, larger degrees a decrease of pain.  The
break-even degree is exactly $n$.  Conditioned on having reached a point with
some edge, adding one for that connection, the average degree is
$ n + \frac {1} 2$. The product rule for determining active pain means that
when the degree is $n-1$ the resulting increase of pain overweighs the
resulting decrease of pain when the degree is $n+1$. The is because
there is a reciprocal relationship between the 
the weight sent to a vertex and the active pain level, and these weights relate
 to the degree of that vertex.  If a weight of $p$ is
sent from $x$ to a vertex $y$ with a  passive pain level
of $r$, the active pain level at $x$ by choosing $p$ is  $rp$. On the
other hand, if a weight of $q$ is sent to a different vertex $z$ with passive
 pain level $w$, and
 the active pain levels equate, we have $ rp= qw$ or $w= \frac {rp} q$.
  A closely related analogy
is the fact that $\frac 1 {n-1} + \frac 1 {n+1} > \frac 2 n$ and
$\frac 1 {n-1} \cdot \frac 1 {n+1} > \frac 1 {n^2}$; the effect is stronger
for~$\frac 1 {n-k} $ and~$ \frac 1 {n+k}$ when $1< k < n$.  But to exploit this
influence sufficiently we needed a variety of possible degrees below the
average. We found success at $n=5$, after failing at $n=2$ and~$n=3$. At~$n=4$,
some preliminary work suggested some difficulty in formulating a proof.  We
suspect that it can be done with $n=4$, but not as nicely as with~$n=5$.
                    
In the next section we describe the colouring rule. In the third section we
show that this colouring rule is paradoxical, given a stochastic structure and
analysis. In the fourth section, we present a computer program confirming that
stochastic analysis.  In the fifth section we apply the colouring rule to a
problem of local optimisation and show that solutions for small enough
approximations of the colouring rule cannot be measurable with respect to any
proper finitely additive extension.  In conclusion we consider related
problems.

\section{A Probabilistic Colouring Rule} 

Let $G$ be the group freely generated by $T_1, T_2, T_3, T_4, T_5$, and let
$X=\{ -1, 1\}^G$.  For any $x\in X$ and $g\in G$, $x^g$ stands for the $g$
coordinate in $x$. With $e$ the identity in $G$, the $e$ coordinate of $x$ is
$x^e$.  There is a canonical right group action on $X$, namely
$g(x) ^h= x^{gh}$ for every $g,h\in G$. We use the canonical product topology
and probability measure, that giving $2^{-n}$ for every cylinder determined by
particular choices of~$-1$ or $+1$ for any $n$ distinct group elements. With
this Borel probability measure the group $G$ is measure preserving.
 
 \begin{definition}
   The {\em graph} of $X$ is the directed subgraph of the orbit graph of the
   action of~$G$ on $X$ induced by the edge subset
   $$\{ (x, T_i x) \ | \ x^e = +1\} \cup \{ (x, T_i^{-1} x) \ | \ x^e = -1\}.$$
   We orient the graph of $X$ by placing arrows from $x$ to all five of the
   $T_i(x)$ if~$x^e=+1$ and arrows from $x$ pointed to all five of the
   $T^{-1} _i(x)$ if $x^e=-1$.
\end{definition}
   
 The subset of $X$ where $G$ does not act freely has measure zero.  Without
 loss of generality, we will be interested only in those orbits of $G$ and
 connected components of the graph of $X$ where $G$ acts freely.
   
\begin{definition} 
  We define $S(y) := \{ x\ | \ y= T_i^{x^e }(x)\}$ and define $ |S(y)|$ to be
  the {\em degree} of $y$ (the number of neighbours in the graph with arrows
  pointed to $y$).
\end{definition}

The graph of $X$ involves two independent structures of arrows. Each point $x$
has a passive and active role, an active role in one structure and a passive
role in the other. The active and passive roles alternate. We are interested in
that alternation, moving from a point in its passive role to its neighbours in
their active roles, and from a point in its active role to its neighbours in
their passive roles. Every point has an active role, namely a connection to
five different points in their passive roles. The degree of a point concerns
its passive role.  Not every point has a passive role, meaning that they are of
degree zero. The points of degree zero play indirectly a key role in the main
argument.

Every point $x\in X$ has a colour in
$\Delta (\{1,2,3,4,5\} ) \times ([0,1] \times_{z\in S(x)} [0,1]_z)$ where the
dimension of the last part of the colouring is equal to the degree of $x$. The
first part, $\Delta (\{1,2,3,4,5\} ) $, a four-dimensional simplex, we call the
{\em active} part of the colour. The $[0,1] \times_{z\in S(x)} [0,1]_z$ part we
call the {\em passive} part of the colour.
 
What is the colouring rule, which we call ${\bf Q}$?
 
Usually the word "cost" is used to describe a function that should be
minimised. With this example, we prefer the word "pain", because it represents
a situation that should be avoided.  For every direction $i\in \{ 1,2,3,4,5\}$
we define the {\em active pain} for~$x$ to be~$v_i \cdot r_x$ where $v_i$ is
the first coordinate of the passive colour of $y= T^{x^e} _i(x)$ and $r_x$ is
the coordinate corresponding to $x\in S(y)$ in the passive colour of
$y= T^{x^e}_i(x)$. The rule for the active colour of $x$ is to choose those
directions where the active pain is minimised. If more than one are minimal,
then any convex combination of the minimal directions is allowed. The quantity
of the active colour of $x$ given to the $i$ coordinate (in the direction of
$T_i$ or $T_i^{-1}$) is called the {\em weight} given in the direction $i$ or
toward~$y= T^{x^e}_i(x)$.
 
The first part of the passive colouring is called the {\em passive pain}.  The
rule for the first part of the passive colouring is as follows.  Whenever the
sum of the active colours in $S(y)$ moving toward $y$ is less than
$1+2^{-11} $, then~$v=0$ is required by the rule $\bf Q$ for the first passive
coordinate.  If that sum is more than $1+2^{-11} $, then~$v=1$ is required by
the rule ${\bf Q}$.  And if the sum is exactly $1+2^{-11} $ then any value in
$[0,1]$ is acceptable for~$v$.

The rule for the $x\in S(y)$ coordinate of the passive colour is very simple,
it is the copy of the $i$ coordinate of the active part of the point $x$
pointed toward~$y$ such that $y= T^{x^e}_i(x)$.

It is now clear from the colouring rule $\bf Q$, why the $T_i$ and their
inverses should remain measure preserving with any finitely additive
extension. A critical aspect of the colouring rule uses that from any $y$ all
the $x$ such that $x\rightarrow y$ are treated equally, e.g. their weights are
summed without prejudice. The same holds for the active part of the colour,
that each of the five directions are treated equally.

First, it is easy to show, with AC (for uncountable families of sets as there
are uncountably many group orbits), that there is some non-measurable solution
to the rule $\bf Q$ valid almost everywhere. By AC we can choose in each orbit
where $G$ acts freely a special point $x$ to correspond to $e\in G$.  Classify
each point $y$ in the orbit containing $x$ according to the length of the word
in $G$ needed to move from~$x$ to~$y$. Choose a direction from $y$ that
involves a word of one length greater and allowed by the coordinate $y^e$.  As
there is only one possibility for a direction from $y$ corresponding to a word
of one length less (and no such possibility if $y=x$), there will be always an
option to satisfy this requirement. Because one always chooses an arrow from a
point with a shorter word to one with a longer word, it is not possible for two
chosen arrows to be aimed toward the same point.  The end result will be a
colouring satisfying the rule $\bf Q$ where there is no pain, passive or
active.
 
Later, we show that if the rule ${\bf Q}$ is satisfied then all points $y$
where the weight sent to $y$ is greater than $1+\frac 1 {10^{11} }$ is
contained in a Borel set of measure zero.  In this way a kind of paradox is
witnessed by the active part of {\bf any} colouring satisfying the rule without
any additional application of theory. The active part of the colouring can be
seen as a distribution of $1$ in ten different directions such that at least
$\frac 1 {2^10}$ of the space receives no weight at all.  We show, however,
that outside of a set of Borel measure zero, no point receives weight more than
$1+ \frac 1 {2^{11}}$.  This can be seen as a kind of paradox, a measure
preserving flow where the flow out ($1$) is greater than the flow in (no more
than $(1-\frac 1 {2^{10}} ) (1+\frac 1 {2^{11}})$.  Furthermore we will show
that the same kind of paradoxical behaviour holds if the rule $\bf Q$ is
followed to a sufficiently small approximation.

Given what will be proven later, we show that any such colouring generates a
measurable $G$-paradoxical decomposition. From \cite{ST1}, we proved it
suffices to have a finite partition of $X$ (generated from the colouring,
actions of the group, and the Borel sets) for which no proper finitely additive
extension can make all partition members measurable.  First approximate the
weights in all directions by integer multiples of $\frac 1 N$ that add up to
$1$ (according to the $N$ different intervals of values between $0$ and $1$ and
the ten different directions) so that the whole space is broken into finitely
many parts according to the values given to the ten different directions of the
colouring (toward the $T_i$ if $x^e=+1$ and toward the $T_i^{-1}$ if
$x^e = -1$).  If $N$ is large enough, the $N$ copies are shifted in this way
(and invariant measurability is assumed), the total weight coming into the
vertices will remain less than the total weight coming out of the vertices.
After defining $N$ different partitions from this, at least one of them cannot
have a proper finitely additive extension for which all partition members are
measurable. Hence by \cite {ST1} we can generate from this partition two Borel
measurable sets of different measure that are $G$-equi-decomposable.

\section{Paradoxical Colouring}
The goal of this section is to complete the proof of  the following theorem:  
    
\begin{thm}\label{thm:paradox}
The colouring rule ${\bf Q}$ is paradoxical. 
\end{thm}

We have shown already that there is some way to satisfy the rule. To complete
the proof of Theorem~\ref{thm:paradox}, we will show that satisfaction of the
colouring rule implies that the set where the passive and active pain is
positive is contained in a Borel subset of measure zero.  To show that $\bf Q$
is paradoxical, it suffices to show that the set of points where the passive
pain is equal to $1$ is contained in a set of measure strictly less than
$\frac { 2^{-11} }9$.  That would be enough to show that the average weight
moving inward toward all points is strictly less than
$$(1-\frac { 2^{-11} }9- 2^{-10} ) (1+2^{-11}) + 10 \frac { 2^{-11} }9=
1-2^{-21} -\frac {2^{-22}} 9.$$

\subsection{Chains} 

\begin{definition}
  Given that $x\rightarrow y$, meaning that $y= T_i x$ if $x^e=+1$ or
  $y= T_i ^{-1} x$ if~$x^e=-1$, the {\em chain} generated by $x\rightarrow y$
  are all the points $z$ in the graph of $X$ that can be reached from $x$
  without going through $y$ and involve alternating arrows, meaning that if $z$
  is an odd distance from $x$ then the passive role $ z \leftarrow $ connects
  $z$ to $x$ and if $z$ is an even distance from $x$ then the active role
  $ z\rightarrow$ connects $z$ to $x$.
\end{definition}
  
A chain involves an alternating process of using the active and passive
colouring functions of the points.  If $x\rightarrow y$ then $x$ distributes
weights to four other points~$z_1, z_2, z_3, z_4$. In turn, depending on their
degrees, there are further directed edges $x^* \rightarrow z_i$ (or in the rare
case that the degree of each $z_i$ is $1$, no further directed edges).  One
could see a chain as a quarter of a $G$ orbit; the choice of seeing $y$ as
passive or active, and the choice of moving in the $x$ direction rather than in
the direction of the potentially other $x^*$ with $x^* \rightarrow y$.

\begin{definition}
  A {\em $0$-level terminating point} of a chain generated by $x\rightarrow y$
  is any vertex of degree $1$ of positive even distance from $y$.  A {\em $1$st
    level terminating point} of the chain generated by $x\rightarrow y$ is some
  vertex $x^*$ of the chain such that $x^* \rightarrow z$ is a directed edge of
  the chain, $z\not= y$, and $z$ is terminating of level $0$ (equivalently has
  degree $1$).  If $i\geq 2$ is even then an {\em $i$-level terminating} point
  of the chain is some vertex $z$ such $x^*$ is a terminating point of $i-1$ or
  less for every $x^*$ such that $x^* \rightarrow z$,~$x^*$ is further from $y$
  than $z$, and furthermore there is at least one such $x^*$ that is a
  $(i-1)$-level terminating point.  If $i\geq 3$ is odd then an {\em $i$th
    level terminating} point of a chain is some vertex $x^* $ of the chain such
  that $x^* \rightarrow z$ is the last step in the alternating path from $y$ to
  $z$ and $z$ is a terminating point of level~$i-1$.  A terminating point is a
  vertex that is a terminating point of some level.  A chain generated by
  $x\rightarrow y$ is {\em terminating} if $x$ is a terminating point, and its
  terminating level is the terminating level of $x$. If the chain generated by
  $x\rightarrow y$ is not terminating, then we say that the chain and the edge
  $x\rightarrow y$ is non-terminating. The non-terminating part of a
  non-terminating chain is the non-terminating chain with its terminating
  points removed.
\end{definition}

\noi {\bf Remark:} A subchain of a terminating chain may not be a terminating
chain. A subchain of a non-terminating chain may be a terminating chain.  We
could have a terminating chain generated by $x\rightarrow y$ with
$x\rightarrow z_1$, $x$ terminating of level~$1$,~$z_1$ terminating of
level~$0$, and $x\rightarrow z_2$ with $x^* \rightarrow z_2$ generating a
non-terminating chain for some $x^*\neq x$.  Likewise $x\rightarrow y$ could be
non-terminating, $x\rightarrow z$ with $z$ of degree three or more, with
$x^*\rightarrow z$ generating a terminating chain for some $x^* \neq x$.
Likewise if $x\rightarrow y$ is terminating it does not imply that $y$ is a
terminating point for all chains it may belong to, as there could be at least
two edges $x^*\rightarrow y$ with~$x^*\neq x$ such that $x^*\rightarrow y$ is
not a terminating chain.

\begin{lem}\label{lem:terminating}
  If $x\rightarrow y$ generates a terminating chain, then in any colouring of
  $X$ that satisfies the colouring rule ${\bf Q}$, the active pain level at $x$
  is zero, meaning that if the passive pain level at $y$ is positive then no
  weight is given by $x$ toward $y$.
\end{lem}
\begin{proof}
  We prove the lemma by induction on the level of the terminating points; we
  claim that any terminating point of even level experiences no passive pain
  and any terminating point of odd level experiences no active pain.  Suppose
  that $x^* \rightarrow y^*$ is part of the chain and $y^*$ is terminating of
  level~$0$. As~$y^*$ is degree one, it is not possible for $y^*$ to experience
  passive pain, since the maximal weight sent to $y^*$ is at most $1$. By
  choosing all weight to $y^*$ there is no resulting active pain, and since the
  rule $\bf Q$ requires that $x^*$ minimise the active pain, $x^*$ cannot
  experience active pain in any direction. We notice that such an $x^*$ is a
  terminating point of level $1$. We continue with the induction assumption.
  Suppose $y^*$ is a terminating point of even level $i$ with $x^*$ the vertex
  such that $x^*\rightarrow y^*$ and the path from $x$ to $y^*$ passes through
  $x^*$.  If $y^*$ was experiencing any passive pain, all the points
  $\hat x \rightarrow y^*$ such that~$\hat x \neq x^*$ would give zero weight
  to $y^*$, since by induction they all experience no active pain. And with
  only one vertex $x^*$ possibly giving weight to $y^*$, it is impossible for
  $y^*$ to experience any passive pain, a contradiction. But then~$x^*$ does
  not experience any active pain either, because it could put all weight
  toward~$y^*$, with the result of no active pain.
\end{proof}

We describe a structure essential to our following stochastic
arguments. Instead of looking at some $x$ according to its topological location
in $X$, we think of $x$ as a member of a chain. Given that $x$ sends the weight
$p>0$ to~$y$ with~$t>0$ the passive pain level at $y$, we consider what the
passive pain levels $t_i$ must be at the four $z_i$ with $x\rightarrow z_i$ for
all $i=1,2,3,4$ so that the active pain levels for each of the five choices are
equal (through satisfying the rule ${\bf Q}$ ).  If $p_i$ is the weight sent
from $x$ to~$z_i$, we must have the equations~$p_i t_i = pt$ for each
$i=1,2,3,4$.  With the equations~$t_i = \frac {pt} {p_i}$, we also consider the
degrees of the $z_i$ and how these quantities continue in further branches in
the chain generated by $x\rightarrow y$. We want to show that almost
everywhere, given~$t>0$, the pain values in further stages must be
unbounded. Since these values cannot exceed $1$, we have shown that positive
passive pain happens only in a subset of measure zero.  It is a kind of reverse
engineering, determining what pain values must exist as implied by the rule
${\bf Q}$.  In this analysis we do not focus on directional choices determined
by the $e$ coordinate; we look instead on the degrees of the vertices of odd
distance to $x$ (even distance from $y$). We use that the probability of
degree~$k$ is~${9\choose {k-1}} 2^{-9} $. It does not follow the
formula~${10\choose {k}} 2^{-10} $ because we condition on the existence of a
particular edge coming toward the vertex.

An important first step toward the main argument is to eliminate all
terminating chains from the analysis and look at only non-terminating chains
and their non-terminating parts.  The following Lemma~ does this.

\begin{lem}\label{lem:nonterm}
  The probability that $x\rightarrow y$ generates a non-terminating chain is
  approximately $\hat q=.991603 $.
\end{lem}
\begin{proof}
  Because of the homogeneous structure to the space, there is a recursive
  formula for the value of $\hat q$.  Given that $x \rightarrow z$ and
  $z\neq y$ and $x^* \rightarrow z$ with $x^* \neq x$,~$\hat q$ is also the
  probability that $x^*$ is not a terminating point of the chain. The
  probability that $z$ is a terminating point is $(1-\frac {\hat q}2) ^9$,
  hence the probability of it not being a terminating point is
  $1- (1-\frac {\hat q}2) ^9$. For $x$ to be a non-terminating point each of
  the four such~$z$ must fail to be terminating points. Therefore $\hat q$ is
  the root of the polynomial $q= (1- (1-\frac q2) ^9)^4$.  Applying Wolfram
  Alpha, the largest root of this polynomial strictly less than $1$ is
  approximately ~$\hat q=.991603$.
\end{proof}

By its definition, we notice that
$\hat q = \sum _ {1\leq j_i\leq 9} \prod_{i=1} ^4 { 9\choose j_i} (\frac {\hat
  q} 2) ^{j_1} (1-\frac {\hat q}2)^{9-j_i}$.

In what follows, we will assume that all terminating points are removed so that
the probability distribution on the degrees in a chain follow the binomial
expansion applied to $\frac {\hat q}2$ and $1-\frac {\hat q} 2$ (instead of
$\frac 12$ and $\frac 12$) and then conditioned to the probability~$\hat q$.
Because a terminating point is a terminating point of some finite level, and
the terminating points of a fixed finite level define an open set, the
non-terminating points form a closed subset of $X$. Likewise the space of
non-terminating chains is a compact space with a probability distribution
determined by the special value of~$\hat q$. Along with a choice for a weight
of $p$ given by some $x$ to $y$ at the start of the chain, this topology
defines a collection of Borel sets on the space of chains.

\begin{definition}
  Given a non-terminating chain generated by $x \rightarrow y$, any
  \mbox{$p\in(0,1)$}, and any colouring $c$ of that chain satisfying the rule
  $\bf Q$ with $p$ the weight of $x$ given to~$y$ and $1$ the passive pain at
  $y$, define $t( x\rightarrow y, c, p)$ to be the supremum of the active pain
  in the colouring $c$ in that chain.
\end{definition}
  
Because we are concerned with the ratio to the quantity $1$, e.g., the passive
pain at $y$ could be some very small positive $v$, we allow for values above
$1$, although strictly speaking there can never be pain, passive or active,
above the level of $1$.  By definition, $t(x\rightarrow y, c,p)$ is no less
than the active pain at $x$, which is $p$.

\begin{definition} 
  Define $u (x\rightarrow y,p) $ to be the infimum of $t(x\rightarrow y, c, p)$
  over all the colourings $c$ satisfying the rule $\bf Q$.  The value
  $u(x\rightarrow y, p)$ is called the {\em chain minimum} with respect to
  $p$. This value can be infinite and we will show that for all positive $p$ it
  is almost everywhere infinite.
\end{definition} 
  
It is straightforward that $u(x\rightarrow y, p)$ is increasing in $p$.
\begin{prop}\label{prop:infinite}
  The function $u(x\rightarrow y, p)$ is infinite for all~$p>0$ and almost all~$x\rightarrow y$.
\end{prop}
 
Proposition~\ref{prop:infinite} implies Theorem~\ref{thm:paradox}.  As both
passive and active pain cannot exceed~$1$, Proposition~\ref{prop:infinite}
implies that the places where the pain, passive or active, is positive is a set
of measure zero.  To prove Theorem~\ref{thm:paradox}, it suffices to show this
for all $p> \frac 1 {10}$, as any point with passive pain~$1$ must be next to
some vertex pointed toward it with weight more than $\frac 1 {10}$.
                  
Let's look at the chain generated by $x\rightarrow y$, where $x\rightarrow z_i$
for $i=1,2,3,4$ and $x^*$ is another vertex where $x^* \rightarrow z_i\neq y$
for some $i$. Assume that $p$ is the weight given by $x$ toward $y$. The
function $u(x\rightarrow y, p)$ is discovered by equalising the
$u(x^* \rightarrow z_i,p^*)\cdot v_i$ over all the choices for weights $q_{i} $
from the $x^*\neq x$ to the various~$z_i$ and by the weights $p^*$ sent to the
$z_i$ from the $x^*$ and the corresponding induced passive pain levels $v_i$
for the $z_i$ following the equality $v_i q_i =p$. Notice that some
$u(x^* \rightarrow z_i,p^*)$ could be infinite, in which case the $p^*$ could
be zero.  Either~$p$ is strictly more than this common equal value, in which
case $p= u( x\rightarrow y, p)$, or~$p$ is less than or equal to this common
equal value, in which case $u(x\rightarrow y, p)$ is that common value.  This
follows by monotonicity, that dividing by $q$ is strictly decreasing in $q$ and
that $u(x\rightarrow y, p)$ is strictly increasing in $p$.  The same holds true
for all the $x^* \rightarrow y^*$ further in the chain.  This means that the
$u(x^*\rightarrow y^*, p^*)$ times the induced passive pain levels, when these
values are finite, form a super-martingale through the minimising process (with
potentially decreasing future values).

\begin{definition}
  Given that $u(x\rightarrow y, p)$ is finite, we define the colouring
  resulting from finding equality at each step (as described above) the {\em
    chain minimiser}. This holds if the minimum is reached with $p$ at the
  start or any other location on the chain, with the chain minimiser finding
  the minimum for the following part of the chain.  If $u(x\rightarrow y, p)$
  is infinite, then the chain minimiser is the result of the same process of
  finding equality, but with $p^*$ replacing $u(x^*\rightarrow y^*, p^*)$ in
  the calculations.
\end{definition}

\begin{lem}\label{lem:A}
  The value of $u(x\rightarrow y, p)$ and the chain minimiser used to define it
  are Borel measurable functions of the future degrees and $p$.
\end{lem}
\begin{proof}
  For every $i$ let $t_i( x\rightarrow y, c, p)$ be the maximum active pain of
  any $x^* \rightarrow y^*$ of distance $i$ or less from $x\rightarrow y$ with
  the colouring $c$ and $p$ the weight given by~$x$ to~$y$.  Let
  $u_i(x\rightarrow y,p)$ be the corresponding minimum over the various $c$.
  As a function of $p$, the degrees, and $c$, the $t_i$ is continuous.  Notice
  that $u_i (x\rightarrow y,p)$ are non-decreasing functions and is always less
  than or equal to $u (x\rightarrow p)$ for every $i$. Notice also that if
  there were a sequence of $p_i$ converging to $p$ and the sequence of
  $u_i (x\rightarrow y, p_i)$ were converging to a value strictly less than
  $u( x\rightarrow y, p)$ then the colourings $c_i$ associated with these
  solutions would have a subsequence converging pointwise to some colouring $c$
  with $t (x\rightarrow y, c, p) $ strictly less than $u( x\rightarrow y, p)$,
  a contradiction.  Hence where it is finite, $u(x\rightarrow y,p)$ is the
  pointwise limit of increasing continuous functions, hence it is Borel
  measurable.  Return to definition of the chain minimiser at the start.  If
  $u(x\rightarrow y,p)$ is infinite, then because it is defined only on the
  $p^*$ of the next stage, the chain minimiser there is Borel measurable.
  Returning to $x\rightarrow y$ and the functions $u(x^* \rightarrow y^*, p^*)$
  from the next stage continuations, because they are Borel measurable the
  function where the equality holds for the different weights is Borel
  measurable and the set where this equality is obtained is a Borel measurable
  set. Since the equalities are unique solutions and the inverse image of a
  Borel measurable set is a Borel measurable set, the colouring that define
  these equalities are also Borel measurable. We proceed by induction on the
  stages.
\end{proof}

\begin{definition}
  Define $u(p)$ to be the greatest lower bound of all $r$ such that $r$ is
  greater than $u(p, x\rightarrow y)$ for some set of $x\rightarrow y$ of
  positive measure.
\end{definition}
  
In the above process of equalising, we would like to minimise using the
function~$u(p^*)$ instead of the $u(x^* \rightarrow y^*,p^*)$.  Of course the
actual process of pain minimisation could look very different because
$u(x^* \rightarrow y^*, p^*)$ may be much larger than $u(p^*)$. But if we
minimise with this assumption, we obtain a result which is not higher than the
proper result almost everywhere.  Our goal is to show that with this
assumption, by applying the rule of minimising the $u(p^*) \cdot v^*$ level on
the next stage, and applying this evaluation stage after stage, the resulting
limit superior of the product is infinite almost everywhere. It follows that
the real value must be infinite almost everywhere also.

Unfortunately we do not know the function~$u$ explicitly, and ultimately we
will show that it is infinite for all positive values; therefore the function
$u$ cannot be of practical use.  So instead of minimising according to $u$, we
introduce some new function $w$ to replace the function $u$, with $w$ finite
everywhere.

There are two useful facts from the replacement of
$u(x^* \rightarrow z_i, p^*)$ by the uniform functions $u(p^*)$ or $w(p^*)$
(should the former be finite). We explain this with the function $w$, which
ultimately we will use.

First, if one aims to minimise active pain levels according to $w(p^*)$ times
the passive pain level of $z_i$ on the next stage (where $p^*$ is the weight
given by some $x^*$ to~$z_i$), because $w$ is a increasing function, it is
sufficient to uses equal weights $p^*$ from all the $ x^*\neq x$ such that
$ x^* \rightarrow z$.

The second fact, following from the equal weights consequence, is by minimising
we do not have to consider any weight $p$ given to some $z_i$ by some
$x^*\not= x$ such that~$p< \frac 1 {100}$. As we already assumed that all the
weights on the other side of~$z_i$ were equal, if $p< \frac 1 {100} $ were
these common weights it means that the weight from~$x$ to~$z_i$ is at least
$.91$.  This means that the weights from $x$ to the other $z_j$ with
$z_j\neq y$ add up to no more than~$0.09$. Therefore the passive pain level at
these other $z_j\neq z_i,y$ are at least~$10$ times that of $z_i$. Assuming
that this is not a terminating chain, there is a $\overline x$ giving weight of
at least $\frac 1 {10}$ to one of these $z_j$. From the monotonicity of $w$ the
chain $\overline x \rightarrow z_j$ receives at least ten times the level of
the chain $\hat x \rightarrow z_i$.  This means that we could redistribute the
weights coming from $x$, giving more to the other $z_j$ and less to $z_i$ with
a reduction in the level as determined by the function~$w$.  We discovered from
our computer calculations that this bound of $0.01$ could be replaced by
$0.055$.

\begin{definition}
  Let ${\cal C}$ be the collection of non-terminating chains with their
  roots. Let ${\cal P}$ be the set of infinite paths in the chains starting at
  their roots.
\end{definition}
    
The set ${\cal C}$ should be understood as the choices of degrees at each stage
and in each position.  The set ${\cal C}$ is a compact set, as the removal of
terminating points is the removal of an open set.  The probability $\hat q$
(approximately $.991603$) determines whether or not a potential edge exists,
and non-terminating implying that after the removal of the terminating points
there is an infinite continuation in each of the four directions
$x^*\rightarrow y^*$ where $x^*$ is closer to the root than $y^*$.  Based on
the choices of degrees, we have a canonical conditional probability
distribution on the collection ${\cal C}$ as determined by $\hat q$.

Whether we use the function $u$, some other function $w$, or leave it as a
sequence of multiplying by $p$s and dividing by $q$s, we should notice the
connection to entropy.  Assuming that we are in the chain generated by
$x_0\rightarrow y_0$, a member $\omega$ of ${\cal P}$ is a sequence of
$(y_0 , x_0 , y_1 , x_1 , \dots)$ such that
$x_0 \rightarrow y_0 , x_0 \rightarrow y_1, x_1\rightarrow y_1, x_1\rightarrow
y_2, \dots$.  If $\omega$ is such an infinite sequence inside a chain
$C\in {\cal C}$, and weights $p_i$ are given to the various
$x_i\rightarrow y_i$ and weights $q_i$ are given to the various
$x_{i-1} \rightarrow y_{i}$, and assuming a normalised passive pain of $1$ at
$y_0$, the active pain at some $x_N\rightarrow y_N$ in the sequence is
$\frac {\prod_{i=0}^N p_i (\omega)} {\prod_{i=1} ^N q_i (\omega) }$. This will
be combined with a probability distribution on these sequences not dissimilar
to that used to define entropy. 

\begin{definition}
  For any path $\omega \in P$ and $i$ define $\phi_i (\omega)$ as
  $$\phi_i (\omega):=\sum_{k=0} ^i \log ( p_k(\omega)) - \sum _{k=1} ^i \log
  (q_k(\omega)).$$
  We  say that $\omega\in C$ if the path $\omega$ is contained in
  $C$.
\end{definition}
                     
Notice that we want to show that, for almost all $C \in {\cal C}$,
$$\lim\sup_{i\rightarrow \infty} \max _{\omega \in C} \phi_i (\omega)=
\infty.$$
It doesn't really matter if we replace the last $p_i$ with a finite
$u(p_i)$ or $w(p_i)$, as long as there is fixed ratio by which $p_i$ cannot
differ from either $u(p_i)$ or $w(p_i)$. This leads to the following
definition.
                
\begin{definition}
  For any $\omega\in P$, we define $\phi^w_i(\omega)$ as
  $$ \phi^w_i(\omega) := \log (w(p_i (\omega))) + \sum_{k=0} ^{i-1} \log (
  p_k(\omega)) - \sum _{k=1} ^i \log (q_k(\omega)).$$
\end{definition}
                 
Given that the sequences $p_i (\omega) $ and the $q_i(\omega) $ are from the
chain minimisers of some $C\in {\cal C}$, we will show that the ratio of
$\log (w(p_{i+1} ) - \log (q_i) + \log (p_i)$ to $\log (w(p_i)) $ tends to
increase in expectation.  To make sense of this tendency, we need to define a
probability distribution on the ${\cal P}$. We have already a probability
distribution on~${\cal C}$, the set of chains, determined by the probabilities
for the future degrees. But to get a probability distribution on ${\cal P}$, we
need to extend it to a conditional probability defined on each chain. We define
that probability distribution using the function $w$ and the process of
equalisation, as described above.
                           
\subsection{The Function $w$}

 We have to define a function $w:[0,1)\rightarrow [1,\infty)$.  The function
$w(x)$ is not far away from $\frac x {1-x}$, so it is easiest to represent it
as $w(x) = \frac x {1-x} f(x) $ with a function $f: [0,1] \rightarrow [\frac 45
, \frac 85]$ such that $f$ is in $C^2$.

Where does the function $w$ come from?  We consider a non-terminating chain
generated by $x\rightarrow y$. We assume that the passive pain at some point
$y$ is normalised at $1$, the probability coming from~$x$ to~$y$ is $p$, and
all vertices are non-terminating of degree $5$.  This means that, if $p=\frac
15$, the resulting passive pain of the entire chain has the value of $1$ and
the active pain $\frac 15$.  Following from~$x$ to the four other vertices on
the other side from~$y$, we assume that the weights are distributed evenly.
This means that from $x$ to the four points~$z_1, z_2, z_3, z_4$, the weight is
$\frac {1-p} 4$. The passive pain $v$ at each of these four places~$z_i$ has to
satisfy~\mbox{$\frac {1-p} 4 v = p$} or~$v=\frac {4 p} { (1-p)}$. In the
definition of the function $w$, we will ignore the initial multiple of $4$ and
also drop the $4$ from other places where it is superfluous.  We continue with
weights $\frac {3+p} {16}$ in the further directions from each of the $z_i$.
The weights at the next stage are $\frac { 13 -p} { 64}$.  The recursive
calculation, a generating function, for the active pain in the limit becomes
$$  \frac {p } {1-p}  \cdot \frac {4(3 + p)} {13-p} \cdot \frac { 4(51+p) } {
  205- p}  \cdot \frac { 4( 819+p)}  { 3277-p} = 
\frac {p } {1-p} \prod_{i=0} ^{\infty}\frac { 4(\frac {16^{i+1} -1} 5+p)}
{\frac {4\cdot 16^{i+1} +1} 5 -p}.$$  
  
   There is a problem with this definition for the function $w$. The above
infinite product is based on the idea that continuation with all points of
degree $5$ is the proper way to represent what happens in general when the pain
levels converge to finite levels.  However we claim a slight tendency for these
pain levels to approach infinity, meaning that the function we need should be a
slight distortion of the above.  The above function works well for the values
$p$ that are used most commonly in the equalisation process.  As stated above,
all such $p$ are greater than $.55$, however usually they are the $p$ between
$\frac 1 {10}$ and $\frac 1 2$.  The values outside of this range need to be
altered with only minimal change for the values within this range.  For the
function $f$ with $ w(p) = \frac p {1-p} f(p)$ we define
         $$f(p) =  e^{ - 1_{ [ 0, \frac 1 {5} ]}(p) \frac {20} {27}  (\frac 1 {5} -p)^3}   \cdot 
         e^{ - 1_{ [ \frac 1 2 , 1 ]}(p) \frac {1} {4} ( p-\frac 12 )^3 } \cdot
\ \frac {4(3 + p)} {13-p} \cdot \frac { 4(51+p) } { 205- p} \cdot \frac { 4(
819+p)} { 3277-p} \cdots,$$ meaning that the function is reduced for the
extremes of $p\in [ 0, \frac 1 5]$ and $p\in [\frac 12 , 1]$.  The function is
altered so that it remains in $C^2$.

\subsection{ The Equalisation Process}
\label{sec:eqprocess}

Let's look at the chain generated by $x\rightarrow y$, where $x\rightarrow z_i$
for $i=1,2,3,4$ and $x^*$ is another vertex where $x^* \rightarrow z_i\neq y$
for some $i$. Assume that $p$ is the weight given by~$x$ toward $y$ and $j_i+1$
is the degree of $z_i$.  Instead of equalising the
$u(x^* \rightarrow z_i,p^*)\frac 1 {q_i} $ over all the choices, we equalise the
$\frac {w(\frac {1-q_i} {j_i} ) } {q_i}$. This can be performed, as $w(p)$ can
be calculated easily, while the $u(x^* \rightarrow z_i,p^*) $ is known only
through understanding the infinite structure.  In what follows, this is what we
call the equalisation process, not the process described above of equalising
the $u(x^* \rightarrow z_i,p^*) $.

We have to relate the equalisation process to the chain minimisers. This is
accomplished by the following lemmas.

\begin{lem} \label{lem:D}
  (a) For every choice of $k=1, \dots , 9$ and $x$, the second derivative of
  $-\log(1-\frac {1-x} k) +\log ( f(\frac {1-x} k)) $ is negative.
  (b) For every choice of $x$, $\frac 16 \leq \frac {d(\log (f(x)))} {dx}
  <\frac {5} {9} $.
  (c) The function $xw(x)$ is convex in $x$.
\end{lem}

\begin{proof} (a) We group the terms by $\log(\frac { k-1+x } k ) - 1_{ [ 1-
\frac k {5} , 1 ]} \frac { 20} {27} (\frac 1 {5} -\frac {1-x} k)^3 - 1_ {[ 0,
\frac k {2} ]} \frac 14 (\frac {1-x} k - \frac 12) ^3 + \log ( 4(3 + \frac
{1-x} k) )$ and then the rest, followed by the pair $- \log {13-\frac {1-x} k}
+ \log ( 4(51+\frac {1-x} k))$ and so on.  Taking the first derivative of the
first three gives
  $$ -\frac {1} {k-1+x}  -\frac 1 k 1_ { [ 1- \frac k {5} , 1 ]}(x) \frac {20} 9 (\frac 1 {5} - \frac {1-x} k)^2 
+ \frac 1 k 1_ {[ 0, \frac k {2} ]} \frac 34 (\frac {1-x} k - \frac 12) ^2 -
\frac 1k\ \frac 1 {3 + \frac {1-x} k}.$$ Taking the second derivative gives
  $$ \frac {1} {(k-1+x)^2}  -\frac 1 {k^2}  1_ { [ 1- \frac k {5} , 1 ]}(x) \frac {40} 9 (\frac k {5} - \frac {1-x} k) - \frac 1 {k^2} 1_ {[ 0, \frac k {2} ]} \frac 32 (\frac {1-x} k - \frac 12) 
  - \frac 1{k^2} \ \frac 1 {(3 + \frac {1-x} k)^2}. $$ Where the two special
restrictions apply are mutually exclusive. In the region $x\not\in [ 0, \frac k
{2} ]$ the minimum occurs at $x=1$, for the quantity $\frac 1 {k^2} (1-\frac 8
9 - \frac 1 9) =0$. Where $x\in [ 0, \frac k {2} ]$ the result is at least
$\frac 1 {k^2} ( 1 -\frac 34 - \frac 1 9) >0$. Moving to the first pair of
following terms, taking the first derivative gives $\frac 1 { 13-x}+ \frac 1
{x+52}$. The second derivative is $\frac 1 {(13-x)^2} - \frac 1 { (51+x)^2}$,
which is also positive. And the same follows for the rest of the pairs.

  (b) The second derivative of $\log (f(x))$ is negative throughout and hence
the derivative reaches its maximum at $x=0$ for the quantity $\frac 4 {45}+
\frac 1 3+ \frac 1 {13} + \frac 1 {51} + \dots < \frac {5} {9} $ and its
minimum at $1$ for $-\frac 3 {16} + \frac 14 + \frac 1 {12} + \frac 1 {52} +
\dots > \frac 16$.

(c) We need the second derivate of $ x\log (w(x)) = x \log (x) - x \log (1-x) -
1_{ [ 0, \frac 1 {5} ]} (x) \frac{20}{27} (\frac 1 {5} -x)^3 x - 1_{ [ \frac
12, 1 ]} (x) \frac{1}{4} (x-\frac 1 2)^3 x+ x \log (4(3 + x)) - x\log (13-x) +
x\log (51+x) - x\log (205-x) + \dots $.  The first derivative is $$ - 1_{ [ 0,
\frac 1 {5} ]} (x) \frac { 20 } {27} (\frac 1 {5} -x)^3 + 1_{ [ 0, \frac 1 {5}
]} (x) \frac {20} 9 (\frac 1 {5} -x)^2 x - 1_{ [ \frac 12, 1 ]} (x) \frac{1}{4}
(x-\frac 1 2)^31_{ [ \frac 12, 1 ]} (x) \frac{3}{4} (x-\frac 1 2)^2$$
            $$  \log (x) -  \log (1-x)    +  \log  (4(3 + x)) - \log (13-x) + \log (51+x)  - \log (205-x) \dots  + $$
          $$   1+ \frac x {1-x} + \frac x { 3+x} + \frac x{13-x} + \frac x {51+x}  +
             \frac x {205-x} + \dots .$$ The second derivative is $$ 1_{ [ 0,
\frac 1 {5} ]} (x) \frac {40} 9 (\frac 1 {5} -x)^2 - 1_{ [ 0, \frac 1 {5} ]}
(x) \frac {40} 9 (\frac 1 {5} -x)x- 1_{[ \frac 12 , 1]} \frac 32 (x - \frac 12
)^2 - 1_{[ \frac 12 , 1]} \frac 32 (x - \frac 12 ) x$$
          $$ 
           +\frac 1 x + \frac 1 {1-x} +\frac 1 {3+x} + \frac 1 {13-x} + \frac 1
{51+x} + \frac 1 {205-x} + \dots $$
             $$  \frac 1 {(1-x)^2 } + \frac 3 {(3+x)^2} +    \frac {13}  {(13-x)^2 }  +  \frac {51}  {(51+x)^2}  +\frac {205}  {(205-x)^2} +  \dots .$$  
                 The negative terms are dominated by $\frac 1 x$ when $x\leq
\frac 1 {5}$ and by $\frac 1 {(1-x)^2}$ when $x\geq \frac 12$. 
\end{proof}

\begin{definition} 
  Define $ g_k (x)$ to be
  $ \frac 1 {x+ k -1} + \frac 1 k \frac { f' (\frac {1-x}k ) } { f (\frac
    {1-x}k ) }$, which is the negative of the first derivative of
  $-\log(\frac {1-x} k ) +\log ( f(\frac {1-x} k)) $. From Lemma~\ref{lem:D} we
  know that $ g_k (x)$ in decreasing in $x$.
\end{definition}
             
\subsection{The Stochastic Process}
\label{sec:stochprocess}

\noi To define a stochastic process on ${\cal P}$, we need a probability
distribution on paths. Before we do that, to use the computer analysis
effectively we need to perform a reduction of the system to discretely many
values.  We alter slightly the chain minimiser of every chain.  We start by
rounding down the initial weight $p_0$, the weight from $x_0$ to $y_0$, to the
highest value of $\frac k {20,000}$ less than or equal to it (for~$k=0$ or $k$
a positive integer).  From $x_0$ to the four different potential $y_1$ we can
increase the weights, so that they add up exactly to $1-p_0$. For the weights
from various~$x_1$ to some $y_1$, with the weight from $x_0$ to $y_1$, they add
up to at least $1+ 2^{-11}$.  We round down these various numbers $p_1$ to the
highest value of $\frac k {20,000}$ less than or equal to~$p_1$ for $k$ a
non-negative integer. Also to simplify the analysis, those rounded down
numbers, which always add up to at least $1$, are normalised to add up to~$1$.
They must add up to at least $1$ because by rounding down one can reduce the
quantity by no more than $\frac 1 {20,000}$, there are at most $9$ such
numbers, and $\frac 9 { 20,000}$ is less than $2^{-11}$.  We require only that
this rounding down process is done in a Borel measurable way. We can continue
this process to all stages, always reducing the value of the products.  After
this alteration, if we show that the limit superior of of the maximum of
$\phi^w_i$ or $\phi_i$ within a chain is infinite almost everywhere, the same
is true before this alteration.  Indeed, we show that any colouring function
obeying the rule ${\bf Q}$ can be altered in this way so that the limit
superior is infinite almost everywhere.
             
             The following can occur in that analysis: there is a $p_i$ with
$p_i= \frac { 19,999} {20,000}$ followed by some passive point $z$ of degree
$2$, meaning that there is only one $p_{i+1}$ following in that direction. As
less than $\frac 1 {20,000}$ can distributed to $z$, it follows that this
$p_{i+1}$ must be more than $\frac { 19,999} {20,000}$. As there are no other
such $p_{i+1}$, it will be rounded down to exactly $\frac { 19,999}
{20,000}$. The sum of $\frac { 19,999} {20,000}$ with the weight to $z$ will be
strictly less than~$1$.  Notice that this situation cannot occur, because with
less than $\frac 1 {20,000}$ sent to $z$ in one direction and at least $1$ in
the other direction there is significantly less than~$1+ \frac 1 {2^{11}}$
weight toward $z$. The result would be that $p_{i+1}$ should be exactly $1$,
all the next $q_{i+1}$ should be zero, and the process jumps immediately to an
infinite value. But we will not exclude this possibility from the
analysis. Instead of letting the process jump immediately to infinity we will
give it some very high but finite value. This is because the following analysis
involves a Markov chain with bounds for the variance on each stage. In the
effort to show that a process should be infinite almost everywhere, we ignore
such situations where it may go to infinity in one step. To include this in the
analysis would greatly complicate it.
                       
Next there are some facts about any $q_1, q_2, q_3, q_4$ chosen to satisfy
the equalisation process: namely, that
\[\frac {w( \frac {1-q_i} {j_i} )} {q_i} = \frac
  {w( \frac {1-q_j} {j_j} )} {q_j},\]
for all $i$ and $p + q_1+ q_2+ q_3 + q_4=1$.
The obvious fact is that $j_i = j_k$ implies that $r_i = r_k$ and $j_i < j_k$
implies $r_i > r_k$. The not so obvious, based on the many calculations, is the
following lemma.
\begin{lem} \label{lem:J}
  The following holds for all $p=\frac n {20,000}$, where
  $n$ is a positive integer less than 20,000: \newline
  (a) if $j_i > j_k$ for some $i,k$, then $q_i < \frac 1 3$; \newline
  (b) if $j_i = j_k+1$, then $|q_i - q_k| \leq \frac 14$; \newline
  (c) if $j_i >j_k>1$, then $q_l + q_m +p > \frac 2 9 $ for $\{ l,m\} \cap \{
  i,k\} = \emptyset$. 
\end{lem}
      
The proof of Lemma~\ref{lem:J} is confirmed by the computer. The discussion of
the nymerical calculations is deferred to Section~\ref{sec:numerical}.
     
\begin{lem} \label{lem:B}
Let $0< p < 1$ and let $q=(q_1, \dots, q_4)$ be a solution to the equalities
$$\frac {w( \frac {1-q_i} {j_i} )} {q_i} = \frac {w( \frac {1-q_j} {j_j} )} {q_j},$$
for all $i$ and $p + q_1+ q_2+ q_3 + q_4=1$ with $p = \frac k {20,000}$ for
some integer $k$ with $0 < k < 20,000$ and $j_1, \dots , j_k$ positive integers
between $1$ and $9$ inclusive. Then, there exists some $s\in \Delta (\{ 1,
\dots , k\})$ such that for all $q^* = (q^*_1, \dots, q^*_k)$ with all $q^*_i$
positive such that $p+ q^*_1+ \dots + q^*_k=1$ it follows that
$$\phi _s (q^*) := \sum_{i}   s_i     \   \log \big(  \frac {w (\frac
  {1-q_i^*} {j_i} ) } { q^*_i} \big)  \geq
\phi_s (q) = \sum_{i}  s_i\  \log \big(   \frac {w ( \frac { 1-q_i } {j_i}
  )} { q_i} \big) .$$
Furthermore for every $\epsilon$ there are finitely many values
$q^*_1. \dots , q^*_n$ for $q^*$ such that after rounding up to the nearest
$q^*_k$ the inequality holds up to $\epsilon$.
\end{lem}
   
\begin{proof}
  With Lagrangian multipliers
  $L :=\phi_{ s} (q^*) -\lambda (1-p -\sum_i q^*_i)$, we define $s$ and create
  a critical point for $\phi_s$ at $q^* =q$ by setting
$$\frac 1 {t_i} := -\frac {\partial \log \big( \frac {w(\frac {1-q^*_i} j)}
  {q^*_i }\big)} {\partial q_i^* } (q_i) $$ and defining
$s_i := \frac {t_i} {\sum_j t_j}$. From Lemma~\ref {lem:D}, all the $t_i$ are
positive.  As convergence to the boundary of the domain of $q^*$ (where
$q^*_i =0$ for some $i$) gives convergence of $\phi_{ s}$ to positive infinity,
it suffices to show that there can be only one unique critical point of
$\phi_{ s}$ so defined, namely the $q= (q_1, \dots , q_k)$.  Another critical
point $r= (r_1, \dots , r_k)\not= q$ would have to satisfy the equalities for
some other $\lambda '$.
         
According to Lemma~\ref{lem:D}, the second derivative of
$-\log ( \frac {w(\frac {1-x} j)} x)$ is equal to
$\frac 1 {x^2} - \frac 1 {(1-x)^2}$ plus a positive term, meaning that the
second derivative of $\frac 12 +x$ plus the second derivative of $\frac 12 -x$
is always positive.  As we could switch $q$ and $r$, without loss of
generality, assume that $\lambda' \geq \lambda$.  As the $q_i$ and the $r_i$
sum up to the same quantity~$1-p$, a second critical point is possible only if
there is some $r_i \geq \frac 12$ with~$q_i \leq r_i$ with~$q_i + r_i >1$ and
$r_k \leq q_k$ for all other $k\not=i$.  We can now assume that the difference
between $q$ and $r$ implies that $\lambda' > \lambda$, $r_i > q_i$ and
$r_j< q_j$ for all~$j\not= i$, and of course that $p< \frac 12$.

A second critical point implies that
$$\frac { \frac 1{r_i} + \frac 1 {1-r_i} + \frac 1 {j_i-1+r_i} + \frac 1 j
  \frac {f'(\frac {1-r_i} {j_i}) } {f(\frac {1-r_i} {j_i})} } { \frac 1{q_i }+
  \frac 1 {1-q_i} + \frac 1 {j_i-1+q_i} +\frac 1 {j_i} \frac {f'(\frac {1-q_i}
    {j_i}) } {f(\frac {1-q_i} {j_i})}} $$ is equal to
$$\frac { \frac 1{r_k} + \frac 1 {1-r_k} + \frac 1 {j_k-1+r_k} + \frac 1 {j_k}
  \frac {f'(\frac {1-r_k} {j_k}) } {f(\frac {1-r_k} {j_k})} }{ \frac 1{q_k}-
  \frac 1 {1-q_k} - \frac 1 {j_k-1+q_k} +\frac 1 {j_k} \frac {f'(\frac {1-q_k}
    {j_k}) } {f(\frac {1-q_k} {j_k})}}, $$ for all $i\not= k$ and all of these
ratios are equal to $\frac {\lambda '} {\lambda}$.  We show that a second
critical point is not possible by showing that this is not possible.

We start with the assumption that the $\{ x_l, y_l\} $ are equal to the
$\{ q_l, r_l\}$ for all~$l$ with $y_i > x_i$ for only one $i$ and choose any
$k\not= i $ such that $\frac {(1-y_i) x_k }{(1-x_i) y_k}$ is maximal. Without
loss of generality let $i=0$ and $k=1 $ (with $x_2, x_3$ and $y_2, y_3$ the
other variables).  With their sums equal to $1-p$ and $y_0 > x_0$ it follows
that $\frac {(1-y_i) x_k }{(1-x_i) y_k} >1$, meaning that
$x_1 (1-y_0 ) > y_1 ( 1-x_0 ) $.

We will take two approaches to proving the above equalities are
impossible. Either we will show directly that the equality is not possible or
we will demonstrate that the equality implies that
                  $$\frac { \frac 1 {y_1} + \frac 1{1-y_1} } { \frac 1 {y_0} + \frac 1{1-y_0} }
                 \quad \leq \quad \frac { \frac 1 {x_1} + \frac 1{1-x_1}} {
\frac 1 {x_0} + \frac 1{1-x_0}}. $$ With $ \frac {y_0 (1-y_0) } {y_1
(1-x_1)}>1$, we get the two inequalities:
$$x_1 (1-y_0 ) > y_1 ( 1-x_0 ) \quad  x_0 y_1 (1-y_1) (1-x_0) \geq y_1 x_1
(1-x_1) (1-y_0).$$
Multiplying together gives $x_0 (1-y_1) > y_0 (1-x_1)$ and adding
$y_1 (1-y_0) > x_1 (1-x_0) $ to this inequality we get $x_0 + x_1 > y_0 +
y_1$. This implies that there must be some~$n= 2$ or $n=3$ with $y_n > x_n$, a
contradiction to $y_j < x_j$ for all $j\not=0$.

Let $k = j_0$ and $l = j_1$. We use that $g_j (x)$ is decreasing in $x$ for any
choice of $j$. 
                     
{\bf Case 1, $k < l$:} Notice that the situation where
$y_i=q_i$ is included in this case, since $y_0> \frac 12$ implies that $k <l$.
If $k+1=l$ we get $|x_0 - x_1| \leq \frac 14$ in both cases of $x_0 = q_0$ or
$x_0 = r_0$ from Lemma~\ref{lem:J} , since $x_0 > x_1$, $y_0 > x_0$ and $y_1 <
x_1$.  So regardless of the values of $k$ and $l$, we conclude that $b:= g_k
(x_0) > d:= g_l (x_1) $. Likewise we define $a:= g_k (y_0)$ and $c:= g_l (y_1)
$ with the the result $b>a$ and $c>d$.  To demonstrate that the above
inequality implies the impossibility of a second critical point, and using $bc
> ad$, it is sufficient to show that $$d(\frac 1 {y_0} + \frac 1 {1-y_0} )+ a
(\frac 1 {x_1} + \frac 1 {1-x_1} ) \leq c (\frac 1 {x_0} + \frac 1 {1-x_0} ) +
b (\frac 1{y_1} + \frac 1 {1-y_1}).$$ By the inequalities $b>a$ and $c> d$ this
is implied by $$d(\frac 1 {y_0} + \frac 1 {1-y_0} )+ b (\frac 1 {x_1} + \frac 1
{1-x_1} ) \leq d (\frac 1 {x_0} + \frac 1 {1-x_0} ) + b (\frac 1{y_1} + \frac 1
{1-y_1}).$$ Now using $b>d$ and that $y_1 < x_1< \frac 12 $ implies $\frac
1{y_1} + \frac 1 {1-y_1} > \frac 1 {x_1} + \frac 1 {1-x_1}$ and likewise $y_0+ 
x_0 > 1$, $y_0 > x_0$ and $y_0 > \frac 12$ implies $\frac 1 {y_0} + \frac 1
{1-y_0} > \frac 1 {x_0} + \frac 1 {1-x_0}$, it suffices to prove that $$\frac 1
{y_0} + \frac 1 {1-y_0} + \frac 1 {x_1} + \frac 1 {1-x_1} \leq \frac 1 {x_0} +
\frac 1 {1-x_0} + \frac 1{y_1} + \frac 1 {1-y_1}.$$

We separate into two parts, to show that
$ \frac 1 {1-x_0} + \frac 1 {y_1} \geq \frac 1 {1-y_0} + \frac 1 {x_1}$ and
$\frac 1 {x_0} + \frac 1 {1-y_1} \geq \frac 1 {y_0} + \frac 1{1-x_1} $. To deal
with the first part, after clearing the dominators one gets equivalence to
$(1-y_1) (1-x_1) (y_0 - x_0) \geq x_0 y_0 (x_1-y_1)$. This inequality follows
from $y_0 + y_1 <1$, $x_0 + x_1 < 1$, and $y_0 + y_1 > x_0 + x_1$. The other
part reduces to the same inequality, after clearing the dominators.

{\bf Case 2, $k=l$:} This is broken down into two cases: {\bf
Case 2A, $x_0 \geq x_1$} and {\bf Case 2B, $x_0> x_1$}. In the former case, we
have something of the form $\frac A B= \frac CD$ where $C> A$ and $ B> D$,
which is impossible. In the latter case, we have the same inequalities of Case~1.
                     
{\bf Case 3, $k>l$: } This is broken down into two cases. In both cases, since
$g_k (x_0) > g_k (y_0)$, they bring down the fraction on that side.  Therefore,
for the sake of contradiction, we assume that
                       $$\frac {\frac 1{y_0} + \frac 1 {1-y_0} }   {\frac 1 {x_0}  + \frac 1 {1-x_0} }\geq 
               \frac {\frac 1{y_1} + \frac 1 {1-y_1} +\frac 1 { y_1 -1 +l}+ e}
{ \frac 1 {x_1} + \frac 1 {1-x_1} +\frac 1 { x_1 -1 +l}+ e} ,$$ where $e$ is an
upper limit for the negative of the derivative of $ \log (f (\frac {1-y_1} l
))$ (which is larger than when $y_1$ is replaced by $ x_1$).
                     
{\bf Case 3A, $l=1$:} We show that
$$\frac { \frac 1{y_0} + \frac 1 {1-y_0} } {\frac 1 {x_0} +\frac 1 { 1-x_0}}
\geq \frac { \frac 1{y_0} + \frac 1 {1-y_0} } {\frac 1 {x_0} +\frac 1 {
    1-x_0}},$$ leading to the above contradiction.  From $ e\leq \frac 59$ it
suffices to show that
                      $$(\frac 1{y_0} + \frac 1 {1-y_0} )(\frac 1 {x_1} + \frac 59) \leq 
                     (\frac 1{x_0} + \frac 1 {1-x_0} ) (\frac 1 {y_1} + \frac
59).$$ That $\frac 1 {1-x_0} \ \frac 1 {y_1}< \frac 1 {1-y_0}\ \frac 1 {x_1} $
follows from the choice of $x_1$ and $y_1$.  We use that $x_0< \frac 13$
implies that $y_0 > \frac 23$. It also implies that $\frac 1 { y_1 x_0}$ is
greater than $\frac 1 { x_1 y_0} + \frac 5 9 \frac 1{ 1-y_0}$ and of course
that $\frac 5 {9x_0} $ is greater than $\frac 5 {9y_0}$.

{\bf Case 3B, $l\geq 2$:} We use that $\frac 1 { y_1 + l
-1}< 1$ and $e\leq \frac 5 {18}$, so that their sum is no more than $\frac {23}
{18}$.  We show that $$\frac {\frac 1{y_0} + \frac 1 {1-y_0} } {\frac 1 {x_0} +
\frac 1 {1-x_0} }\geq \frac {\frac 1{y_1} + \frac 1 {1-y_1} +\frac {23}{18}} {
\frac 1 {x_1} + \frac 1 {1-x_1} +\frac {23} {18} } $$ is impossible, or with
cross multiplication that
                   $$(\frac 1{y_0} + \frac 1 {1-y_0} ) ( \frac 1 {x_1} + \frac 1 {1-x_1} +\frac {23} {18} ) \geq 
                   (\frac 1 {x_0} + \frac 1 {1-x_0}) (\frac 1{y_1} + \frac 1
{1-y_1} +\frac {23} {18} ) $$ is impossible. From the choice of $x_1$ and $y_1$
we get $\frac 1 {y_1 (1-x_0)} > \frac 1 {x_1 (1-y_0)}$.  From Lemma~
\ref{lem:J} we have $x_0 + x_1 < \frac 79$, with of course $x_0\leq \frac 1 3$
and $x_1 \leq \frac 7 9$.  From the choice of $x_1$ and $y_1$ we have $\frac 1
{x_0 y_1} > \frac {1-x_0} {x_1 (1-y_0) x_0}$.  From $x_0 < \frac 13$ and
$1-x_1= x_0 + p + x_2+ x_3$ and $p+ x_2 + x_3\geq \frac 2 9$ we have $\frac
{1-x_1} {x_0} > \frac 5 3$. From $1-x_0= p + x_1+ x_2 + x_3$ we have $\frac
{1-x_0} {x_1} > \frac 9 7$.  Together we get $\frac {21} {45} \frac 1 {x_0 y_1}
> \frac 1 {(1-y_0) (1-x_1)}$.  From the choice of $x_1$ and $y_1$ we get that
$\frac { 1-y_0} {y_1} > \frac { 1-x_1} {x_1} > \frac 9 7$. So we can write $
\frac 13 \frac 79 \frac {23} {18} \frac 1 { x_0 y_1} > \frac {23} {18} \frac 1
{1-y_0}$.  Notice that $ \frac 13 \frac 79 \frac {23} {18}= \frac {161} {486}$.
With $\frac {161} {486} + \frac { 21} {45} <1$ we can conclude that $ \frac 1
{x_0 y_1} > \frac 1 {(1-y_0) (1-x_1)} + \frac {23} {18} \frac 1 {1-y_0}$.  It
is only left to show that $$\frac {23} {18} \frac 1 {x_0} + \frac 1{(1-x_0)}
(\frac 1 {y_1} + \frac 1 {1-y_1}) + \frac 1 {x_0 (1-y_1) } > \frac {23} {18}
\frac 1 {y_0} + \frac 1 { x_1 y_0} + \frac 1 {y_0 (1-x_1)}.$$

{\bf Case 3Bi, $l\geq 2$, $x_1\geq \frac 1 2$:} From Lemma~
\ref{lem:J} we get $x_0<1 - \frac 12 - \frac 2 9 = \frac 5 {18}$ and therefore
$y_0 > \frac {13} {18}$, so $\frac {23} {18} ( \frac 1 {x_0}- \frac 1 {y_0})>
\frac {23} {18} \ \frac {144} {65}$.  From the definition of $y_1$ and $x_1$ it
holds that $\frac {y_1} {x_1} < \frac 5 {13}$ With $x_1< \frac 7 9$ and $1-x_0
< y_0$ we get $ \frac 1{(1-x_0)} \frac 1 {y_1} -\frac 1 { x_1 y_0}> \frac {144}
{65}\ \frac 97$.  We have $\frac 1 {x_0 (1-y_1) } > \frac {18} 5$ and with
$x_1$ no more than $\frac 7 9$ and $y_0 > \frac {18} {23}$ we have $\frac 1
{y_0 (1-x_1)} < \frac {23} { 4} $. The quantity $ \frac 1{(1-x_0)} \frac 1
{1-y_1} $ is at least $1$. With $(\frac {23} {18} +\frac 9 7) \ \frac {144}
{65} + \frac {18} 5 + 1> \frac {23} { 4} $, the case is settled.

{\bf Case 3Bii, $l\geq 2$, $x_1\leq \frac 1 2$:} With $x_0<
\frac 13$ and $y_0> \frac 23$ we have $\frac {23} {18} ( \frac 1 {x_0}- \frac 1
{y_0})> \frac {69} {36}$.  From the definition of $y_1$ and $x_1$, it holds that
of $\frac {y_1} {x_1} $ is less than $\frac 1 2$. With $x_1 \leq \frac 12$ and
with $1-x_0< y_0$, we get $ \frac 1{(1-x_0)} \frac 1 {y_1} -\frac 1 { x_1 y_0}>
2$. The quantity $ \frac 1{(1-x_0)} \frac 1 {1-y_1} $ is at least~$1$.  With
$x_1$ no more than $\frac 12$, we have $\frac 1 {x_0 (1-y_1) } > 3$ and $\frac 1
{y_0 (1-x_1)} < 3 $, and the case is settled.

Finally, notice that there are finitely many possibilities for the choice of
degrees $(j_1, j_2, j_3, j_4)$ and $p=\frac k {20,000}$. For each such choice,
any $q^*_k$ small enough so that~$s_k \frac {1-q^*_k} { q^*_k j_k} $ alone
exceeds the total expectation of
$\sum _{i=1} ^4 s_i \frac {1-q _i} { q_i j_i} $ suffices for the lowest value
needed.  The approximation by $\epsilon$ follows from the fact that $\frac 1 x$
is uniformly continuous when positive $x$ is bounded from below.  \end{proof}

Given any fixed $C\in {\cal C}$, we need to determine a conditional
probability distribution $P(\cdot \ | \ C )$ on the paths $\omega$ that belong
to $C$.  We start with the root $x\rightarrow y$ of the chain $C$, and call
$x=x_0$ and $y=y_0$. Let $z_1, z_2, z_3, z_4$ be the points such that~$x_0
\rightarrow z_i$, with $j_1, j_2, j_3, j_4$ the positive integers between $1$
and $9$ such that $j_i +1$ is the degree of $z_i$. Let $q_1, q_2, q_3, q_4$ be
the weights from $x_0$ to the $z_i$ that solve the equalisation process.  The
probability of moving in the direction from $x_0$ to $z_i$ is the quantity
$s_i$ as determined by Lemma~\ref{lem:B}. The probability of moving from $z_i$
to~$x_{i,k} $ is~$\frac {p_{i,k}} {\sum _{l=1}^4 p_{i,l} }$, where $p_{i,l}$ is
the weight given to $z_i$ by $x_{i,k}$.  We continue in this way defining the
probability in terms of these products.
                
With a conditional probability distribution defined on each
chain, and a probability distribution defined on the chains, we need to extend
this to a probability distribution defined on ${\cal P}$. To do this, we use
the expectations on the $P(\cdot \ |\ C )$. In order for this to make any
sense, the conditional values we get on the chains must be Borel measurable. We
get that from Lemma~\ref{lem:A} It could be noticed that if there is symmetry
to the way the quantities are rounded down, for any given choices for~$j_1,
j_2, j_3, j_4$ and $i$ the expectation for the $p_{i,k}$ will be equal. We
do not use this in the proof.

\begin{cor}
  According to the above probability distribution, if the $q_i$ and $p_{i,k}$
  are from the chain minimiser, the expectation of
  $\log (w(p_{i,k} ) - \log ( w(p) ) + \log (p) - \log (q_i) $ is positive.
\end{cor}

\begin{proof}
  It follows from Lemma~\ref{lem:B} and Lemma~\ref{lem:D} part c (from the
  fact that equal quantities of $p_{i,k}$ for a given $i$ defines a critical
  point and from the convexity there is a unique minimiser).
\end{proof}

Recall the definition of $\phi^w_i (\omega)$. We can break this sum into two
parts.  We can perform the equalisation process at each step, and for any
sequence $p_0, q_1, p_1, q_2, p_2, $ $\dots, p_{i-1}, q_i, p_i $ corresponding
to a path $\omega$ define a sequence of triples $(p_0, q_1', p_1'),$ $ (p_1,
q_2', p_2'), \dots , (p_{i-1}, q_i', p_i')$ where the $q_i'$ are defined by the
equalisation process and the $p_i'$ are defined by equality for each weight
going to the same point $z$ in the chain. We can break down the expression of
$\phi^w_i (\omega)$ into two parts, that involving the triples and the
difference. Call $\overline {\phi^w_i}$ the sum of the part involving the
triples and~$\tilde {\phi^w_i }$ the difference $\phi^w_i - \overline
{\phi^w_i}$.  By the above corollary, we have shown that, conditioned on any 
chain, the expectation of $\tilde {\phi^w_i}$ is positive. Now, we turn to the
other part, the~$\overline {\phi^w _i}$.

We need to show that the $\phi^w_i$ functions are unbounded on almost every
chain in ${\cal C}$.  As we use only $20,000$ many values for the $p_i$, it
suffices to do the same for the $\phi^w_i$. As it does not matter where on the
chain the value of $\phi^w_i$ is maximal, it suffices to show that the
expectation of $\phi^w_i$ goes to infinity on almost all chains.  As the
expectation of $\tilde {\phi^w_i}$ is always positive, attention is drawn to
the $\overline {\phi^w_i}$, the part of the process from the sequence of
triples.
                 
We want to define a Markov chain from the triples that define
$\overline {\phi^w_i}$ and show that it defines a submartingale on this Markov
chain that approaches infinity almost everywhere. However strictly speaking the
triples do not define a Markov chain. The problem is that each system of
weights is determined by the membership of some chain $C$ in ${\cal C}$, and
therefore those weights are determined by the future.  However we can relate
this process to a Markov chain through an inequality.
                       
\begin{definition}
  For every $p=\frac k {20,000}$ and every $(j_1, j_2, j_3, j_4)$ (choice of
  $1\leq j_i\leq 9 $) $i=1,2,3,4$ such that $p+ \sum_{i=1}^4 q_i=1$, define
  $w_1 (p, j_1, j_2, j_3, j_4)$ to be the common value for
  $\frac {w( \frac {1-q_i} {j_ i} ) } {q_i} $ from the equalisation process.
  Define $r(p, j_1, j_2, j_3, j_4)$ to be
  $\log (w_1(p, j_1, j_2, j_3, j_4) - \log (w(p))$.  For every choice
  $ (j_1, j_2, j_3, j_4)$ and every $p= \frac k {20,000}$ for all
  $k=1, \dots , 19,999$,  we define
  $$\overline r ( j_1, j_2, j_3, j_4) := \min_{k=1, \dots, 19,999} w_1 ( \frac k
  {20,000},j_1, j_2, j_3, j_4 ). $$
\end{definition}

As before, $\hat q$ is the probability that a chain is terminating, which we
approximated at $\hat q=.991603$.  For each choice of $1\leq j_1, j_2, j_3,
j_4\leq 9$, we sum up the logarithm of $\overline r ( j_1, j_2, j_3, j_4)$
times the probability $\prod_{i=1} ^4 { 9\choose j_i} (\frac {\hat q} 2) ^{j_1}
(1-\frac {\hat q}2)^{9-j_i}$ and divide by $\hat q$ (to condition on the event
that the chain is not terminating) to get the rate of increase $s$.
 
\begin{lem}\label{lem:K}
  The rate $s$, the conditional expectation of $\overline r$, is at least
  $\frac 1 {1,000}$.
\end{lem}
The proof of Lemma~\ref{lem:K} is done with the help of the
computer. A discussion of these numerical calculations can be found in Section~\ref{sec:numerical}.
            
\begin{prop} \label {prop:wgrowth}
  The process $ \phi^{\overline w}_i$ converges to positive infinity almost
  everywhere. 
\end{prop}
\begin{proof}
  The Markov chain from the $\overline r (j_1, j_2, j_3, j_4)$ is well defined.
  The Kolmogorov inequality states that, if $X_1, X_2,\dotsc$ is a martingale
  starting at $X_0$, then, for~$\epsilon>0$, the probability that
  $\max_{0< i \leq n} |X_i-X_0| > \epsilon$ is no more than the sum of the
  variances of the $X_i-X_{i-1} $ divided by $\epsilon^2$. As only finitely
  many values for $p$ and $(j_1, j_2, j_3, j_4)$ are used, the variances at
  each stage have a uniform bound $B>0$ (determined by the two extremes of
  $j_1= j_2= j_3= j_4=1$ and $p=\frac {19,999} {20,000}$ and
  $j_1= j_2= j_3= j_4=9$ and $p=\frac {1} {20,000}$).  After subtracting the
  $s>0$ a martingale $x_i$ is defined with $X_n = \sum_{i=1} ^n x_i$.  The
  cumulative variance of the process to the $n$th stage is the sum of the
  variances at each stage, which is no more than $nB$.  If the subset where the
  limit superior before removing the $s$ is not infinite has positive measure,
  there must be an $\epsilon >0$ such that for every $n$ the probability that
  $|X_{n} |$ is greater than $\frac {ns}2$ is at least $\epsilon$.  But this is
  not true, since the Kolmogorov inequality says that this probability is not
  greater than $\frac {4nB} {n^2 s^2}$ for every $n$.
\end{proof}

Now we can prove the main result.
   
\begin{prop} \label{prop:1}
  The limit superior of the maximal values of the $\phi_i$ is infinite for
  almost all chains $C\in {\cal C}$. 
\end{prop}
\begin{proof}
  Suppose there is a bound $M$ and a subset $A$ of chains of positive measure
  is such that the highest value of $\phi_i$ for all $i$ in the subset $A$ is
  $M$.  Because the expectation of $ \tilde {\phi^w_i}$ is non-negative, this
  means that in this subset $A$ the expectation of $\overline {\phi^w_i } $
  must be less than $M$. But this is impossible, since
  $\overline { \phi^w _i} $ approaches infinity almost everywhere.
\end{proof}

\section{The Numerical Calculations}
\label{sec:numerical}

The first problem is to calculate the function $w$. As it is defined above, it
is difficult to calculate with great precision because the infinite product
doesn't converge quickly. The influence of each term is approximately
one-fourth of the previous term, and that means to gain accuracy to less than
one-millionth requires the use of around ten terms. After ignoring the
exponential part, we want to convert the infinite product into a few products
followed by a power series. However we notice that the coefficients of
$3$,$13$, $52$, $\dots$ are not easy to work with. We make a simple
substitution, $t= p-\frac 15$, with $t$ now standing for the difference from
the norm of $\frac15$.
   
We start with $\frac {p} {1-p}$. After the substitution $p= t+\frac 15$ we
get $\frac { t+\frac 15} { \frac 45 -t} =\frac { 1+ 5t} { 4 -5t} $. As
multiplying one time by $4$ doesn't change anything, we get $\frac { 1+ 5t} { 1
-\frac 5 4t}$.

Next comes $\frac {3+p} {13-p}$. After the substitution $p= t+\frac 15$ we get
$4\frac { t+3+ \frac 15} { \frac {64} 5 -t}= \frac { 5t+16} { 16 -\frac 54 t}
$. We recognise the pattern
$$ \tilde w (t) := \frac { 1+ 5t} { 1 -\frac 5 4t}\ \frac { 16+ 5t} { 16 -\frac
  54 t}\ \frac { 16^2+ 5t} { 16^2 -\frac 54 t}\cdots.$$ To make accurate
calculations of $\tilde w$, we use the first three products and then change the
rest into the geometric power series. We get
$$ \tilde w (t) := \frac { 1+ 5t} { 1 -\frac 5 4t}\ \frac { 16+ 5t} { 16 -\frac
  54 t}\ \frac { 16^2+ 5t} { 16^2 -\frac 54 t} (1+ \frac {5t} {16^3} ) (1+
\frac {5t} {16^4} ) \cdots $$
   $$(1+\frac 5 {16^3 \cdot 4} t + \frac {5^2} {16^6\cdot 4^2}t^2 + \frac { 5^3} {16^9 4^3} t^3 + \dots ) 
   (1+\frac 5 {16^4 \cdot 4} t + \frac {5^2} {16^8\cdot 4^2}t^2 + \frac { 5^3}
{16^{12} 4^3} t^3 + \dots) \cdots $$ Collecting the $t$ and $t^2$ terms via the
geometric series and including the first $t^3$ term gives a very good
approximation:
    $$ \frac { 1+ 5t} { 1 -\frac 5 4t}\ \frac { 16+ 5t} { 16 -\frac 54 t}\ \frac { 16^2+ 5t} { 16^2 -\frac 54 t}  \big( 1 + \frac5 {3\cdot 4\cdot 16^2 } t+ 
   \frac {100} {16^5 \cdot 9 \cdot 17} t^2 + \frac { 125} { 4^3 \cdot 16^9}
t^3\big).$$ The first $t^3$ term dominates the rest (and true also of the
higher powers of $t$) and so the error is less than $\frac 1 {10^{10}}$. Even
dropping the second power term puts one within~$\frac 1 {10^6}$, which is good 
enough, considering that the rate of expansion is slightly greater than~$\frac
1 {1,000}$.

The first term is easy to calculate with a geometric series: it is
$\frac {5t} {16^3}\ \frac 4 3 = \frac {5t} { 16^2 \cdot 12} $. The second term
comes in two parts. First there are the terms that come directly from the
$t^2$, or
$\frac {5^2t^2} {16^7} \frac {16^2} { 255}= \frac {5^2t^2} {16^5\cdot 255}.$
The rest are products of single powers of $t$. We use that if $a_1, \dots, a_n$
are numbers and we want to calculate $\sum_{i<j} a_i a_j$ we could calculated
instead $\frac 12 ( (a_1+ \dots + a_n) ^2 - a_1^2 - \dots a_n^2)$.  If we want
to calculate $\sum_{0\leq i <j} a b^i $ for some positive $b$ less than $1$ we
get
$\frac 12 ( (\frac a {1-b})^2 - \frac {a^2} {(1-b^2)} ) = \frac { (1+b) a^2 -
  (1-b) a^2} {2(1-b)^2 (1+b) } = \frac {a^2 b} { (1-b)^2 (1+b) }$.  In our case
it is $a=\frac 5 {16^3}$, $b= \frac 14$ and we get
$\frac {25t^2 \frac 14 } { {16} ^6 \frac 9 {16} \frac 54 }= \frac {5 t^2}
{16^5\cdot 9}$.  For the second term we get the sum
$\frac {5 t^2} {16^5\cdot 9} + \frac {5^2t^2} {16^5\cdot 255}= \frac {100 t^2}
{ 16^5\cdot 17\cdot 9}$.
  
The function $\tilde w (t)$ gets converted back to $w(p)$ with the substitution
$t= p-\frac 15$ and the inclusion of the exponentials at the two ends, which appear
on lines~99--110 on Page~32. 


The function \lstinline{generic_thread(j1,j2,j3,j4)} (on Page~31, starting from
line~35) solves for the equalisation process, for each~\lstinline{j1,j2,j3,j4},
and does the bookkeeping for keeping track of the various quantities with which
the statment of Lemma~\ref{lem:J} is concerned. In the final nested loop of the
programme (lines 142--156 on Page~33), the global bounds required by
Lemma~\ref{lem:J} and the calculation of rate~$s$ from Lemma~\ref{lem:K}.

The output of the code confirms the statements of Lemma~\ref{lem:J} and
Lemma~\ref{lem:K}.
\begin{lstlisting}
The value of \hat{q} in Lemma 2: 0.9916

Proof of Lemma 5.
The maximum difference |q_k - q_l| for j_k=j_l+1: 0.22955
The maximum q_k for j_k>j_l: 0.32546
The minimum value of p+q_3+q_4 for j_2>j_1>1: 0.23163

Proof of Lemma 7.
The expectation of the ratios: 0.0010956
\end{lstlisting}

\section{Approximation}
We can define the colouring rule in terms of a problem of local
optimisation. At every point choices are made according to an objective
function, which will be the sum total of three variables corresponding to the
three types of choices that are made, the choice of five weights, the copying
of those weights by adjacent points, and the choice of a passive pain level.
We use the term {\em solution} for a function from~$X$ to the colour set $C$
obeying the rule approximately, so as not to confuse it with "objective
function".

The rule for the active colouring is already phrased in terms of an
optimisation, the minimisation of active pain. As for the passive colourings,
it is easy to make it the result of a minimisation. Let $c(y)$ be the sum total
of weights directed at $y$. Choosing a level of $0\leq b\leq 1$ at $y$ results
in a cost of $ (1-b) \cdot c(y) + b \cdot (1+\frac 1 {2^{11}}) $, with
preference for $b=0$ if $c(y) < 1+\frac 1 {2^{11}}$, preference for $b=1$ if
$c(y) > 1+\frac 1 {2^{11}}$, and any value for $b$ if $c(y) = 1+\frac 1
{2^{11}}$.  The copying of the weight of an adjacent point is done easily by
taking the absolute value of the difference between the weight and the
choice. Approximate copying will be done later in an affine way with finitely
many options when we present the local optimisation again as a Bayesian game.

We can see from its formulation that the invariance of the group $G$ for any
finitely additive extension is necessary for this optimisation problem. At any
point, the weights toward it from different directions are given equal
consideration for determining the passive pain. The same is true for the five
directions involved in the choice of minimal active pain.

Of course for every positive $\epsilon$ there will be a measurable
$\epsilon$-optimal solution where optimality is understood with respect to all
the measurable options.  On the other hand, given a measurable solution, we can
integrate the objective function over the whole space and from the need for the
weights inward to equal the weights outward it follows that expectation of the
objective function will not go below~$\frac 1 9 (2^{-11} $ (from the passive
pain alone).  This does not come close to the $0$ result almost everywhere when
using some non-measurable solutions. Both of these options for understanding
$\epsilon$-optimality are not interesting.

\subsection{Stability}
We are interested in a special kind of $\epsilon$-optimality, which we call
$\epsilon$-stability. For each~$x\in X$, let $t(x)$ be the possible improvement
in the objective function at $x$, keeping the solution for all other $y\neq x$
fixed.  Let $\mu$ be a proper finitely additive extension.  A solution is {\em
  $\epsilon$-stable} (w.r.t. $\mu$) if the $\mu$-expectation of $t(x)$ is no
more than $\epsilon\geq 0$, meaning that there is no finite disjoint collection
$A_1, \dots , A_n$ of $\mu$ measurable sets such that the objection function
can be improved by at least $t_i$ at all points in~$A_i$
and~$\sum_{i=1} \mu (A_i) t_i $ is greater than $\epsilon$.  Another way of
understanding $\epsilon$-stability is that~$X$ is a uncountable space of human
society or molecules, and the solution is $\epsilon$-stable if the gains from
the individual deviations do not add up to an expectation of $\epsilon$.  Our
claim is that there is a positive $\epsilon$ such that no solution that is
measurable with respect to any proper finitely additive extension is
$\epsilon$-stable (and likewise for any $\epsilon^* <\epsilon$). This does not
mean that if the deviations happened simultaneously there would be such an
improvement for all concerned; indeed the result may be worse for all
concerned.

There are two ways that a measurable solution must obey
$\epsilon$-stability. First, the set where there is significant divergence from
optimality must be small. Second, where divergence from optimality exists in a
subset of large measure, that divergence must be small. That can be formalised
in the following way: if a solution is $\epsilon \cdot \delta$-stable, then the
subset where it diverges from optimality by more than $\delta$ cannot be of
measure more than $\epsilon$.

For any $\delta>0$ the rule ${\bf Q}^{\delta}$ applies to a point $x$ if all
three aspects of the colour at~$x$ (choosing weights, copying weights for each
direction separately, and responding with passive pain) are $\delta$-optimal at
$x$ with respect to the rule ${\bf Q}$ and furthermore in its passive role
there is no terminating point $x^*$ of odd level such that $x^*\rightarrow x$
and the weight given by~$x^*$ to~$x$ is more than $\frac 1{10} 2^{-12}$.  The
condition on non-terminating points is a way to ignore the terminating points
and reduce our analysis to the non-terminating points.  By $\delta$-optimal we
mean that an improvement by $\delta$ in each aspect is allowed, but no more. In
this way the rule becomes a closed relation.  That the $\delta$ applies to each
aspect of the colouring separately greatly simplifies the following analysis.

Assume that there is an option to choose $a$ or $b$, $a$ gives a cost of $0$,
$b$ a cost of $(1-\delta)$-optimality for a positive $\delta$ means that
there cannot be more than $\delta$ weight given to $b$, since otherwise by
switching one could gain by more than $\delta$. A choice of exactly $\delta$
for $b$ and $1-\delta$ for $a$ is $\delta$-optimal, because by switching to $a$
only a gain of $\delta$ can be accomplished.

Again we introduce the concept of the stochastic process on non-terminating
points, except that the rule ${\bf Q}$ is replaced by the approximate rule
${\bf Q}^{\delta}$.  The stochastic process is defined only for the
non-terminating points, so that we retain the analysis using the probability
$\hat q$ for non-terminating points.  As before, minimising of the future pain
levels is done with the function $w$ used at every stage.  And as before, the
analysis is almost identical, showing that with near certainty the pain, both
passive and active, must reach unobtainable levels and therefore the assumption
of a significant probability of passive pain at level $1$ is not possible.
There are two main differences however. First, we cannot make this claim for
all positive passive pain levels, as we did for the ${\bf Q}$ rule. If the
passive pain level is small compared to~$\delta$, one could slip away from the
logic of the rule.  Second, we have to re-introduce the influence of the
terminating points, for the same reason, that extremely small pain levels could
be involved. The $w$-process is defined only on non-terminating points, but to
make it apply properly we have to assume that the contributions from
terminating points are sufficiently small.
     
With $\delta$ sufficiently small, the ${\bf Q}^{\delta}$ rule implies that the
quantities directed to a passive point with passive pain of at least $\delta$
must be at least $1+ \frac 9 {20,000} $. After rounding down to quantities of
the form $\frac k {20,000}$ for positive integers $k$, we have the same
structure as before, that the weights toward each passive point add up to $1$
(with those rare exceptions already discussed above).
                            
We must still deal with the terminating points and a subset where
the ${\bf Q}^{\delta}$ might not apply.

\begin{lem}\label{lem:5}
  Let $y$ have passive pain of level $v$, $x$ be a terminating point of odd
  level $n$ in the chain generated by $x\rightarrow y$ with $x$ giving $y$
  weight of at least $\frac 1 {10\cdot 2^{12} } $. Furthermore assume that each
  point between $y$ and the terminating point of level $0$ satisfies the rule
  ${\bf Q}^{\delta}$. It follows that
  $\delta \geq v(\frac {2^{-12}} {10} )^{ {n+1} } $.
\end{lem}
\begin{proof}
  Let $x$ be terminating of level $n\geq 3$.  Because $x$ gives weight of at
  least $\frac 1 {10\cdot 2^{12}}$ to $y$, its active pain is at least
  $v 10\cdot 2^{12}-\delta 10\cdot 2^{12})$ in all directions and therefore the
  terminating point of level $n-1$ next to $x$ has passive pain of at least
  $v 10\cdot 2^{12}-\delta 10 \cdot 2^{12} $ (as the weight given to any other
  point cannot exceed~$1- \frac 1 {10\cdot 2^{12}}$).  The result follows by
  induction, after noticing that a point of $0$ terminating level (degree $1$)
  cannot have a passive pain level of more than $\delta 2^{-11}$.
\end{proof}

\begin{lem}\label{lem:6}
  Let $i$ be odd and let $q_i$ be the probability of a chain $x\rightarrow y$
  having terminating level $i$ (meaning that $x$ is a terminating point of
  level $i$).  Then the probability $q_1$ is less than $\frac 1 {128}$ and the
  probability of $q_{\frac {i-1} 2} $ is less than
  $\frac 1 {128} (\frac 16) ^{\frac {i-1} 2}$.
\end{lem}
\begin{proof}
  The probability that $x\rightarrow z$ and $z\neq y$ is terminating of level
  $0$ is exactly~$2^{-9}$. Since there are four such $z$, $q_1$ is no more
  than~$4\times 2^{-9}$.  Now assume that $x\rightarrow z$ and~$z\neq y$ is a
  terminating point of level $i-1$. There is at least one $x^*$ that is a
  terminating point of level $i-2$ with $x^*\rightarrow z$ and no other
  $\hat x\rightarrow z$ that is non-terminating. The probability is no more
  than~$9\cdot ( 1-\frac {\hat q} 2) ^8 q_{i-2}$, where $\hat q$ is
  approximately $.991603$.  Since this could happen in any one of four places,
  $q_{i}$ is no more than $4\cdot 9\cdot ( 1-\frac {\hat q} 2) ^8 q_{i-2}$. The
  conclusion holds by induction and that
  $4\cdot 9\cdot ( 1-\frac {\hat q} 2) ^8 <\frac 16$.
\end{proof}

The argument that the expectation over the paths $\omega$ in ${\cal P}$ of the
sequences
$\log (p_0(\omega)) - \log (q_1(\omega)) + \log (p_1(\omega) ) - \log
(q_2(\omega) ) + \dots - \log (q_i(\omega) ) + \log (p_i (\omega))$ approaches
positive infinity does not use the ${\bf Q}$ rule, rather holds for any choice
of the sequences $p_0 , q_1, \dots , p_i$. All that was required to define the
stochastic process, and the Markov chain lying within it, is that expectations
for the $p_i$ and $q_i$ values are well defined at each stage.  We could do
this in at least one of two ways.  One way would be to define a unique chain
minimiser with the ${\cal Q}^{\delta}$ rule, show that it is Borel measurable,
and proceed in the same way as before.  Another way would be to work directly
with any finitely additive $G$-invariant measure. We choose the latter way.  To
do it the latter way, we prefer to reformulate the stochastic process with only
finitely many possibilities at each stage.  These choices for the finitely many
values must be independent of the distributions implies by the finitely
additive measure, otherwise we may run into trouble due to the lack of
countable additivity. We are justified in this by Lemma~\ref{lem:B}.  In what
follows, we assume that there are finitely many values for the $q_i$ and $p_i$
and that with this assumption the expectation of $\overline r$ is at least
$\frac s 2$.  To define the stochastic process, we use that the colouring
function is measurable according to any finitely additive process. But implicit
in the probability calculations following the binomial expansion is that the
finitely additive measure is proper.
     
\begin{lem}\label{lem:Q}
  Let $\delta>0$ be smaller than $\frac 1 {20,000}$, let $v>0$ be the passive
  pain level at~$y$, let $p>\frac 1 {20,000} $ be a weight from $x$ to $y$
  satisfying the ${\bf Q}^{\delta}$ rule, let $v_i$ be the passive pain level
  at $z_i$ satisfying the ${\bf Q}^{\delta}$ rule with $x\rightarrow z_i$, and
  let $q_i$ be a weight satisfying the ${\bf Q}^{\delta}$. It follows that
  $\log (v_i)\geq \log (p) - \log ( {q_i} ) - (30,000 )^2\cdot \delta \cdot
  \log (\frac 1 v)$.
\end{lem}
       
\begin{proof}
  The copying of the weight $p$ at $y$ must be within $\delta$ of $p$.  Hence
  the active pain at $x$ in the direction of $y$ must be at least $vp
  -\delta$. If the active pain in the direction of $z_i$ were not at least
  $vp-2001\delta $, there would be a gain of at least $\delta$ by replacing all
  the weight in the $y$ direction over to the $z_i$ direction. As the copying
  of the weight $q_i$ in the $z_i$ direction is within $\delta$, it follows
  that by choosing $q_i$ in that direction the active pain at $x$ is also
  within $\delta v_i$ of $q_i v_i$. We conclude that $v_i$ is at least
  $\frac {vp - 2002\delta} {q_i} $. The rest follows by taking the $\log $ of
  both sides and that $p\geq \frac 1 {20,000}$ and
  $\delta \leq \frac 1 {20,000}$.
\end{proof}

\begin{thm}\label{thm:nostable}
  There is a positive $\gamma $ small enough so that there is no
  $\gamma$-stable solution to the rule $\bf Q$ that is measurable with respect
  to any proper finitely additive extension.
\end{thm}
\begin{proof}
  We assume that the $p$ values have been rounded down to integer multiples of
  $\frac 1 {20,000}$ and that there are finitely many $q$ values that preserve
  the property that the expectation of the
  $\log(p_0 (\omega) ) - \log (q_1 (\omega)) + \dots - \log(q_i(\omega)) + \log
  (p_i(\omega))$ goes to infinity as $i$ goes to infinity.  We will prove that,
  with sufficiently small positive~$\epsilon$ and $\delta$, the subset
  where~${\bf Q}^{ \delta}$ does not hold must exceed $\epsilon$, given that
  the solution is properly measurable. We start with a hypothetical chain
  generated by $x\rightarrow y$ where the passive pain at $y$ is at least
  $\frac 3 4$ and show that this can happen with only a very small probability.
   
  There is a positive integer $N$ such that the probability is at least
  $1- \frac 1 {100} 2^{-12} $ that there is some path $\omega$ with
  $\log(p_0 (\omega) ) - \log (q_1 (\omega)) + \dots + \log (-q_N)$ greater
  than $3$.  As there is a lower bound on all the $ \log (p_i)$, from
  Lemma~\ref{lem:Q}, in a non-terminating chain generated by $x\rightarrow y$
  where $y$ has a passive pain level of at least $\frac 1 2$ there is
  a~$\delta^*$ such that if the ${\cal Q}^*$ rule is followed, then after a
  distance of $N$ the probability is at least $1- \frac 1 {100} 2^{-12} $ that
  a passive pain level of $2$ is reached, (which is impossible).

  The number of vertices of distance $N$ away from a point in a chain of
  length~$N$ does not exceed~$50^N$. So we make positive $\epsilon$ smaller
  than $ \frac {2^{-12}} {100\cdot 50^N}$ and make positive~$\hat \delta$
  smaller than $ \delta^* \frac 12 e^{-NM} $. All that is left is to control
  for the probability that a terminating point of odd level sends more than a
  weight of $\frac 1 {10} 2^{-12} $ to a non-terminating point in the chain of
  distance no more than $N$ from the initial $y$.

  By Lemma~\ref{lem:6}, there is some odd $i$ such that the chances of some
  terminating point of level $i$ or more sending more than
  $\frac 1 {100} 2^{-12}$ to any of these vertices, at most~$50^N$ of them, is
  less than $\epsilon $.  So, we define our $\delta$ to be
  $ (\frac {2^{-12}} {100} )^{i+1} \hat \delta$ and use Lemma~\ref{lem:5} to
  have our pair $\delta$ and $\epsilon$ such that a measurable
  $\delta \epsilon$ stable solution is not possible.
\end{proof}
     
Notice that the last part of the proof incorporates both possibilities of
$x\rightarrow y$ being either terminating or non-terminating.  This proof is
far from optimal in choosing a $\delta$ and $\epsilon$, and we are sure that
this choice can be done much better.

\subsection{A Bayesian Game}

Our interest in paradoxical colouring rules came originally from game theory,
from the desire to show that {\bf all}, not just some, equilibria of a game are
not measurable.  R. Simon \cite{S} showed that there is a Bayesian game which
had no Borel measurable equilibria, though it had non-measurable
equilibria. The infinite dihedral group, an amenable group, acted on the
equilibria in a way that prevented any equilibrium from being measurable.

R. Simon and G. Tomkowicz~\cite{ST} showed that there is a Bayesian game with
non-measurable equilibria but no Borel measurable $\epsilon$-equilibrium for
small enough positive $\epsilon$ and later \cite {ST2} that there is a Bayesian
game with non-measurable equilibria but no measurable $\epsilon$-equilibria for
small enough positive $\epsilon$ where measurable in the above means with
respect to any finitely additive measure that extends the Borel measure and
respects the probability distributions of the players.  These constructions
involved the action of a non-amenable semi-group.

Some background to Bayesian games can be found in \cite{ST2} and \cite{HL}. Of
particular importance is the relationship to countable Borel equivalence
relations.

Let $G=\mathbf F_5$ be the group generated freely by five generators, $T_1,
T_2, T_3, T_4, T_5$ and let $X$ be the Cantor set $\{ -1, 1\} ^G$.  Let $A$ be
the set $\{ a_i^+, a_i^- \ | \ i=1,2,3,4,5\}$ of cardinality~$10$ and let $B$
be the set $\{ b_i^+, b_i^- \ | \ i=1,2,3,4,5\}$ of cardinality~$10$. We assume
that~$A$ and $B$ are disjoint.  Let $C$ be the set $A\cup B$ of
cardinality~$20$ and let~$\Omega$ be the space $X \times C$. Let $m$ be the
canonical probability distribution such that the measure of a cylinder set
defined by
$$\{ x \ | \ x^{g_1}= f_1, \dots, x^{g_l} = f_l \}$$ is equal to $2^{-l}$
for every sequence $f_1, \dots , f_l$ of choices in $\{ -1, 1\}$ and~$g_1,
\dots , g_l$ are mutually distinct.  Define the Borel measure $\mu$ on $\Omega$
by
$$\mu (A \times \{c\}) = \frac {m(A)} {20},$$
for every Borel measurable set $A$ in $X$ and any choice of $c$ in $C$.

There are two players, the active player, called the green player, and the
passive player, called the red player.  An information set for a player is
another term for a member of that player's partition. For every $x\in X$, the
green player has the information set
$$(\{ x\} \times A) \bigcup_{i=1,2,3,4,5} \{ (T_i (x), b_i^-) , (T_i^{-1}
(x) , b_i^+) \}.$$ For every $y\in X$, the red player has the information set
$$(\{ y\} \times B) \cup_{i=1,2,3,4,5} \{ (T_i (y) , a_i^-), T_i^{-1}
(y), a_i^+ ) \}.$$ Notice that each information set is of cardinality $20$ and
for both players these sets partition the space. To identify the information
set of the player, the green player is {\bf centred at} $x$ if
$\{ x\} \times A$ is half of its information set and the red player is {\bf
  centred at}~$y$ if $\{ y\} \times B$ is half of its information set, meaning
that if nature chooses some~$(y, b)$ with $b\in B$ then the green player is
centred at some neighbouring point while the red player is centred at $y$ (and
a symmetric statement can be made if nature chooses some $(y,a)$ with
$a\in A$).  We will also refer to~$(x, a_i^+)$ as~$(x, a_y)$
where~$y= T_i (x)$, $(x, a_i^-)$ as~$(x, a_y)$ where $y= T_i^{-1} (x)$,
$(y, b_i^+)$ as $(y, b_x)$ where~$x= T_i^{-1} (y)$, and $(y, b_i^-)$ as
$(y, b_x)$ where $x= T_i (y)$.

The green player has the choice of $5$ actions, $t_1, t_2, t_3, t_4, t_5$. A
strategy for the green player at any $x$ is a point in the four-dimensional
simplex $\Delta (\{ 1, 2, 3, 4, 5\} )$.

The red player centred at~$y$ has the choice of $2\cdot M^{d(y)}$ actions,
where $M$ is a very large positive integer, size to be determined later.  The
set of actions is
$$ \{ c, u\} \times \prod _{x\in S(y)} \{ m_x\ | \ 0 \leq m_x \leq M-1\}.$$
The symbol $c$ stands for ``crowded'' and $u$ for ``uncrowded''. The choice of
a mixed strategy for the red player is for some point in the $2\cdot
M^{d(y)}-1$ dimensional simplex.

The payoffs for the green player centred at $x$ take place only in $\{ x \}
\times A$, meaning that in the other ten locations the payoff is uniformly
zero. The payoffs for the red player centred at~$y$ take place only in $\{ y\}
\times B $.  It is more restrictive than this. The payoffs for the green player
centred at $x$ take place only in the five locations $\{( x , a_i^+) \ | \
i=1,2,4,5\} $ if $x^e =+1$ or only in the five locations $\{( x , a_i^-) \ | \
i=1,2,4,5\} $ if $x^e =-1$. The payoffs for the red player centred at $y$ take
place only in that subset of $B$ corresponding to the subset $S(y)$ (meaning
only at the $b_x$ with~$x\in S(y)$).  With both players, as each gives the
probability~$\frac 1 {20}$ to each point in its information set, the payoff is
determined by summing over all the points giving equal weight to each. The key
to understanding is that whatever is played by the green player centred at $x$
is done uniformly throughout its information set $(\{ x\} \times A)
\cup_{i=1,2,3,4,5} \{ (T_i (x), b_i^{-}) , (T_i^{-1} (x) , b_i^+) \} $, and the
same is true for the red player centred at $y$ and its information set.

First we define the payoffs for the green player.  We consider what happens to
the green player centred at $x$ when choosing the action $t_i$.  The action
$t_i$ has a payoff consequence only at the point $(x, a_y)$ where $y= T_i^{x^e}
$.  Given that the red player centred at~$y$ chooses $(c, m_x, *)$, where $*$
stands for any choices of $m_{x'}$ for other~$x'\in S(y)$, the payoff to the
green player centred at $(x,a_y)$ is $-\frac {m_x} M$.  . Otherwise for all
combination with~$u$ instead of $c$ the payoff is $0$.

Now we define the payoffs for the red player.  For any $x\in S(y)$, meaning $y=
T_i^{x^e}(x)$, let $t_y$ be the action $t_i$.  First consider a piece-wise
linear convex function $f: [0,1] \rightarrow {\bf R}$, where $f= \max_k f_k$
for some affine functions $f_0, \dots , f_{M-1}$ where $f$ is equal to $ f_k$
on $[\frac k {M}, \frac {k+1} {M}]$.  Let~$s_k^+$ and~$s_k^-$ be defined by
$f_k (0) = s_k^-$ and $f_k (1) = s_k^+$, and the difference in slopes between
consecutive~$f_i$ and $f_{i+1}$ is always at least~$1$.  Define the value of
the actions~$(c, m_x, *) $ played against~$t_y$ at $(y,b_x)$ to be~$s_{m_x} ^+
+ 1$, the value of the actions $(u, m_x, *) $ played against~$t_y$ at $(y,b_x)$
to be~$s_{m_x} ^+ + 1+r$, for any $z\neq x$ the value of the actions $(c, m_x,
*) $ played against~$t_z$ at~$(y,b_x)$ to be $s_{m_x} ^- $, and for any $z\neq
x$ the value of the actions $(u, m_x, *) $ played against~$t_z$ at $(y,b_x)$ to
be $s_{m_x} ^- + 1+r$.

Because the consequence for the red player centred at $y$ by choosing some
$m_x$ for~\mbox{$x\in S(y)$} lies entirely at the point $(y, b_x)$ and is also
independent of the choice for~$c$ or $u$, the red player will chose the
marginal probabilities for~$m_x$ according to~$s_{m_x}^+$ and $s_{m_x}^-$ and
the probability for $t_y$ performed by the green player centred at $x$.  By
the structure of those values, no more than two $m_x$ will be chosen in
equilibrium, and only two adjacent $m_x-1$ and $m_x$ if the probability for
$t_y$ is exactly~$\frac {m_x} {M}$. When the probability for $t_y$ lies
strictly between $\frac {m_x} {M}$ and $\frac {m_x+1} {M}$ then only $m_x$ will
be chosen in equilibrium.

Notice that in equilibrium this game not only approximates the colouring rule
of the previous sections, and it can be done so in a way for which the computer
calculations also apply.  If the green player centred at $x$ chooses the action
$t_y$ with probability $q$, the red player centred at $y$ will mimic with
various combinations of $(c, m_{[ qM]}, *), (u, m_ {[qM]}, *)$ and possibly
with some $(c, m_{qM -1}, *), (u, m_{qM-1}, *)$ if $qM$ is an integer. The cost
for the green player centred at $x$ and with the action $t_i$ will be $\frac 1
{20}$ times the red player's total probability of playing $c$ centred at $y$
times some quantity that is between $ \frac {[qM-1]} M $ and $\frac {[qM]} M$.

To show a lack of an $\epsilon$-equilibrium (measurable with respect to any
proper finitely additive extension) using our previous argument for the lack of
an $\epsilon$-stable solution, we require that the process of copying weights
is done with sufficient precision.  Whatever $ \delta$ worked for the $\epsilon
\delta$-stability argument above, we divide by 3 and declare this to be the
quantity needed for the lack of finitely additive measurable~$\epsilon \delta
/3$-equilibria for this Bayesian game.  We make $M$ be larger than~$\frac 3 {
\delta}$ to insure that there is no inaccuracy up to $\frac {\delta} 3$
resulting from the intervals used. But lastly, we need to know that there is no
relevant distortion from the mixture of the $c$ and the~$u$ coordinates with
the occasional choice of a level $m_j$ that is not a good copy of the actual
weight sent from the relevant point. We need to know that the summation of the
probabilities given to the actions $(c, m_j, *)$ is sufficiently close to the
average value for $m_j$ times the average proportion for $c$ (the product of
expectations from the marginals).  Lets suppose that the level $m_j$ is
incorrect when $m_i$ is the choice closest to the correct choice on the same
side as~$m_j$.  Due to the slopes of the lines defining the payoffs, we know
that the cost of this mistake is at least $\frac {|j-i| ( |j-i| -1)} 2 q_j$
where $q_j$ is the probability of using $m_j$.  We have that the summation over
$j$ of the $ \frac {|j-i| (|j-i| -1) } 2 q_j$ cannot exceed $\delta /3$.  It
follows that~$\frac 1 M \sum_j q_j |i-j|$ cannot exceed $2\delta/M$.  By
choosing~$M$ greater than $\frac 3 { \delta}$, we have the needed accuracy.

\section{Conclusion}
What interested us initially about paradoxical colouring rules was the
connection to the Banach-Tarski Paradox.

\begin{que}
  For all colourings $c$ satisfying the rule $\bf Q$ is there a finite
  partition of the colour space into Borel sets such that the inverse images of
  this finite partition along with the Borel sets ${\cal F}$ and shifts in $G$
  generate a finite partition of $X$ with the Banach-Tarski property, e.g. they
  create two copies of $X$ after shifting by members of $G$?
\end{que}
     
A further issue is raised by the expected value of the non-measurable
solutions. With rule ${\bf Q}$ there exists non-measurable solutions where
optimality is perfect, meaning the pain level of $0$ almost everywhere. And
with all measurable solutions there is an average passive pain level above
$\frac {2^{-11} }9$.

\begin{que}
  Does there exist a problem of local optimisation or a Bayesian game such that
  the optimisation can be accomplished locally or the values can be measured
  globally, but not both simultaneously?
\end{que}

Theorem~\ref{thm:paradox} uses a free non-abelian group of rank 5.  Given the
existence of non-amenable groups without free non-abelian subgroups,
demonstrated by Olshanskii and Grigorchuk, (see [TW], Chapter 12 for the
details) it is natural to ask the following:

\begin{que}\label{que:withoutfree}
  Does there exists a probabilistic paradoxical colouring rule that uses a
  non-amenable group without free non-abelian subgroups?
\end{que}

The idea behind Question~\ref{que:withoutfree} is related to the complexity
behind the proof of Theorem~\ref{thm:paradox}. Recall that two or three free
choices were not enough to obtain a paradoxical rule. So it is natural to
investigate and describe if the required complexity can be forced by generators
that are not independent.

The paradox would be more graphic if passive pain began with $1$ rather
than $1+ \frac 1 {2^{11} }$, meaning that, outside a set of measure zero, a 
colouring satisfying the rule defines a flow where to every point there is no
more than a total of $1$ going inward (and yet in $2^{-10}$ of the space there
is no inward flow).  Could one find a colouring rule with a much stronger
paradoxical effect?  Instead of choosing between the five directions with the
incoming arrows of variable degree, one could assume that toward any point
there are always five arrows coming in but leaving from any point there are
anywhere from $0$ to $10$ arrows. Instead of a rule defined by the avoidance of
pain, the goal might be to obtain pleasure by directing weight toward where
weight is lacking.  If we could show that in general (except for a set of
measure zero) the weights directed toward a point add up to at least $1$, then
the inward flow is at least $1$ but the outward flow is no more than
$1-\frac 1 {2^{-10}}$. Initial investigation suggest that this could have a
stronger paradoxical effect.

The proof of Theorem~\ref{thm:nostable} seems convoluted. Terminating points
and non-terminating points are treated separately, and it would be nice to have
a unified approach.  The problem is that in the calculations behind
Theorem~\ref{thm:paradox}, integrating the effect of terminating points into
the argument would involve a division by $0$ (as we divide by one less than the
degree of the vertex).  Indeed terminating points are such that they need
infinite levels of pain in order to avoid sending all weight toward them. The
present approach is not efficient for establishing a good upper bound for the
$\epsilon$ for which there is no measurable $\epsilon$-stable solution. Again,
a colouring rule with a stronger paradoxical effect is desired.

\begin{que}
  What is the largest positive $\epsilon $ such that there is a probabilistic
  paradoxical colouring rule defined by a local optimisation where the
  objective function is between $0$ and $1$ and there is no $\epsilon$-stable
  solution that is measurable with respect to any proper finitely additive
  extension?
\end{que}

\appendix

\section{Code}
\label{sec:code}

In this section, we present the entire C++ code used to establish the
correctness of Lemma~\ref{lem:J}. Section~\ref{sec:maincode} presents the main
function, which controls the parallel computation of the quantities required for
the proof of Lemma~\ref{lem:J}. Section~\ref{sec:helpercode} includes the code
for the classes and for various helper functions needed for the numerical
calculations. The entire code and compilation instructions can accessed at a
GitHub repository~\cite{GH}.

\subsection{Main Code}
\label{sec:maincode}

\begin{lstlisting}
// Compile command:
// g++ genanalysis_threaded.cpp -o genericanalysis -pthread -std=c++11

#include<iostream>
#include<fstream>
#include<cmath>
#include<thread>

using namespace std;

const long double accuracy = 1e-11L;
const long double lb=0.0L,ub=1e+8L; // Bounds for the function values
const int N=5;        //
const int NoFns=N-1;  // no of functions
const int types=9;
const int jconfigs = 495; // #(a,b,c,d)\in types^4 s.t. a<=b<=c<=d

const int maxC=19999;
const long double M = (long double) maxC+1;
const int minC=1;

long double vcur[maxC+1];
// long double vnext[maxC+1];

long double r;

long double maxdiff1[types+1][types+1][types+1][types+1];
long double maxjumpval[types+1][types+1][types+1][types+1];
long double mintail[types+1][types+1][types+1][types+1];
long double mingen[types+1][types+1][types+1][types+1];

#include "Func1.cpp"
#include "helpers1.cpp"

void generic_thread(int j1, int j2, int j3, int j4){

  Func* Fns[NoFns];
  long double roots[NoFns];
  Fns[0] = new Func_iter(j1,vcur);
  Fns[1] = new Func_iter(j2,vcur);
  Fns[2] = new Func_iter(j3,vcur);
  Fns[3] = new Func_iter(j4,vcur);

  long double mxd=0.0L, mxj=0.0, mnt=1.0L;
  mingen[j1][j2][j3][j4]=ub;
    
  for(int m=1;m<=maxC;m++){
    long double budget = 1.0 - (m/M);

    SimulSolver(Fns,NoFns,roots,budget,lb,ub,accuracy);

    if (j1==j2-1)
      mxd = max(mxd,abs(roots[1]-roots[0]));
    if (j2==j3-1)
      mxd = max(mxd,abs(roots[2]-roots[1]));
    if (j3==j4-1)
      mxd = max(mxd,abs(roots[3]-roots[2]));
    if (j1<j2)
      mxj = max(mxj,roots[1]);
    if (j1<j3)
      mxj = max(mxj,roots[2]);
    if (j1<j4)
      mxj = max(mxj,roots[3]);
    if (j2<j3)
      mxj = max(mxj,roots[2]);
    if (j2<j4)
      mxj = max(mxj,roots[3]);
    if (j3<j4)
      mxj = max(mxj,roots[3]);
    if (j1>1 && j2>j1)
      mnt = min(mnt, (m/M) + roots[2] + roots[3]);

    long double t = (*Fns[0])(roots[0]);
    long double ratio = log(m*t/M) - log(vcur[m]);
    if (ratio < mingen[j1][j2][j3][j4])
      mingen[j1][j2][j3][j4] = ratio;
    
  }
  maxdiff1[j1][j2][j3][j4] = mxd;
  maxjumpval[j1][j2][j3][j4] = mxj;
  mintail[j1][j2][j3][j4] =  mnt;
}

int main(){
  cout.precision(5);
  
  binomcoef = new int[types+1];
  calcbinomcoef(types);

  r=rSolver(N,0.0,0.9,1.0,accuracy); // r should be very close 1
  cout << "The value of r: " << r << endl;

  // for calculating v_1 from innerfn()
  Func_ini f(1);
  // Adjustments for the bottom and top end of the range
  for(int m=minC;m<=maxC;m++){
    vcur[m]=f.innerfn(m/M - 1.0/N);
    //    cout << m << " : " << vcur[m] << endl;
    long double g = 1.0;
    if (m<4000){
      g = (m/M - 0.2);
      g *= (20.0/27.0)*g*g;
      g = exp(g);
    }
    if (m>10000){
      g = (m/M - 0.5);
      g *= -0.25*g*g;
      g = exp(g);
    }
    vcur[m] *= g;
  }

  thread *myth[jconfigs];
  int t_i = 0;
  for(int j1=1;j1<=types;j1++){
    for(int j2=j1;j2<=types;j2++){
      for(int j3=j2;j3<=types;j3++){
	for(int j4=j3;j4<=types;j4++){
	  myth[t_i] = new thread(generic_thread,j1,j2,j3,j4);
	  t_i++;
	}
      }
    }
  }

  t_i = 0;
  for(int j1=1;j1<=types;j1++){
    for(int j2=j1;j2<=types;j2++){
      for(int j3=j2;j3<=types;j3++){
	for(int j4=j3;j4<=types;j4++){
	  if ((*myth[t_i]).joinable()){
	    (*myth[t_i]).join();
	    t_i++;
	  }
	}
      }
    }
  }

  long double mxd=0.0L, mxj=0.0L, mnt=1.0L;
  long double sum = 0.0L;
  for(int j1=1;j1<=types;j1++){
    for(int j2=j1;j2<=types;j2++){
      for(int j3=j2;j3<=types;j3++){
	for(int j4=j3;j4<=types;j4++){
	  mxd = max(mxd,maxdiff1[j1][j2][j3][j4]);
	  mxj = max(mxj,maxjumpval[j1][j2][j3][j4]);
	  mnt = min(mnt,mintail[j1][j2][j3][j4]);

	  long double p = probfn(j1,j2,j3,j4,r);
	  long double w = mingen[j1][j2][j3][j4];
	  sum += w*p*noperm(j1,j2,j3,j4);
	}
      }
    }
  }

  cout << endl << "Proof of Lemma 5." << endl;
  cout << "The maximum difference |q_k - q_l| for j_k=j_l+1: "
       << mxd << endl;
  cout << "The maximum q_k for j_k>j_l: " << mxj << endl;
  cout << "The minimum value of p+q_3+q_4 for j_2>j_1>1: " << mnt
       << endl << endl;
  
  sum /= r;
  cout << "Proof of Lemma 7." << endl;
  cout << "The expectation of the ratios: " << sum << endl;
}  


\end{lstlisting}

\subsection{Helper Functions}
\label{sec:helpercode}

\begin{lstlisting}
// implements function (1-(1-x/2)^{2n-1})^{n-1} - x
// to find the value of r using bisection solver rSolver
long double rfun(long double x, long double n){

  double result = pow((1.0-(x/2.0)),2*n-1);
  result = 1 - result;
  result = pow(result,n-1)-x;
  return result;
}

// bisection solver for rfun() function
long double rSolver(long double n,long double target,long double lb,
                    long double ub,long double accuracy){
  long double mid;
  do{
    mid=(lb+ub)/2;
    //    cout << mid << "   "; 
    long double y=rfun(mid,n);
    // cout << y << endl;
    if (y>=target)
      lb=mid;
    else
      ub=mid;
    }while (ub-lb >= accuracy);
  return mid;
}

// factorial function for n>=1
int factorial(int n){
  int f = 1;
  while (n>1)
    f *= n--;
  return f;
}

// calculates no. of different permutations of j1,j2,j2,j4
// 1 <= j1 <= j2 <= j3 <= j4 <= maxtypes
int noperm(int j1, int j2, int j3, int j4){
  int t[4];
  t[0]=j1;
  t[1]=j2;
  t[2]=j3;
  t[3]=j4;

  int no = 24; // 4!
  int i = 1; // index 
  int rep = 1; // number of repetitions of a repeated value
  while(i<4){
    if (t[i]==t[i-1]){
      i++;
      rep++;
    }
    else{
      no /= factorial(rep);
      i++;
      rep=1;
    }
  }
  no /= factorial(rep);
  return no;
}

int *binomcoef;

void calcbinomcoef(const int n)
{
  binomcoef[0]=1;
  for(int i=1;i<=n;i++){
    binomcoef[i] = binomcoef[i-1] * (n-i+1) / i;
  }
}
  
long double probfn(int i,int j,int k,int l,long double r){ 
  long double result=1.0;
  result *= binomcoef[i];
  result *= binomcoef[j];
  result *= binomcoef[k];
  result *= binomcoef[l];
  result *= pow(1.0*r/2.0,i+j+k+l);
  result *= pow((1.0-1.0*r/2.0),36-i-j-k-l);  // 36 = NoFns * types

  return result;
}
  

// BisectionSolver tries to solve for x in [lb,ub] such that f(x)=target
// It stops when ub-lb < accuracy
// It initially assumes and maintains that f(lb) >= target >= f(ub)
long double BisectionSolver(Func* f,long double B,long double target,
                            long double lb,long double ub,
                            long double accuracy){
  long double mid;
  do{
    mid=(lb+ub)/2;
    long double y=(*f)(mid);
    if (y>=target)
      lb=mid;
    else
      ub=mid;
    }while (ub-lb >= accuracy);
  return mid;
}

// SimulSolver tries to solve nofn functions simultaneously such that
// fn[1](x_1)=fn[2](x_2)=...=fn[nofn](x_nofn) and x_1+x_2+...+x_nofn=B
// The common function value should be in [lb,ub]
// fn[i]s are assumed to be non-increasing
// The equalities are checked within accuracy
// SimulSolver updates argument array roots with corresponding x_i values
void SimulSolver(Func* fn[],int nofn,long double* roots, long double B,
                 long double lb,long double ub,long double accuracy){
  do{
    long double sum=0.0;
    long double mid=(lb+ub)/2;
    for(int i=0;i<nofn;i++){
      roots[i]=BisectionSolver(fn[i],B,mid,0.0,1.0,accuracy); 
      sum+=roots[i];
    }
    if (B-sum >= accuracy)
      ub=mid;
    else if (sum-B>accuracy)
      lb=mid;
    else
      break;
  }while (ub-lb >= accuracy);
  return;
}

\end{lstlisting}

What follows contain the definition of the \lstinline{Func} class that is used to
represent functions $w(x/j_i)$ from Section~\ref{sec:stochprocess}.

\begin{lstlisting}
class Func{
public:
  virtual long double operator()(long double x)=0;
  long double k;
  virtual long double innerfn(long double x)=0;
};

class Func_iter:public Func{
public:
  Func_iter(long double key, long double *vec);
  long double operator()(long double x);
  long double innerfn(long double x);
  
private:
  long double *v;
};

class Func_ini:public Func{
public:
  Func_ini(long double key);
  long double operator()(long double x);
  long double innerfn(long double x);
  
private:
  long double *numerator, *denom, *series;
};

Func_iter::Func_iter(long double key, long double *vec){
  k=key;
  v=vec;
}
  
long double Func_iter::innerfn(long double x){
  static int minindex=maxC+1,maxindex=0;
  if (x>=1)
    return 1e+10;
  int i = (int) floor( M * x);
  return v[i];
}

long double Func_iter::operator()(long double x){
  return innerfn((1.0-x) / k) / x;
}
 
Func_ini::Func_ini(long double key){
  k=key;
  series = new long double[4];
  series[0] = 1.0;
  series[1] = 5.0/3072.0;
  series[2] = 100.0/17.0/9.0/1024.0/1024.0;
  series[3] = 125.0/1024.0/1024.0/1024.0/4096.0;
}

long double Func_ini::innerfn(long double x){
  long double result=series[3];
  for (int i=2;i>=0;i--){
    result *= x;
    result += series[i];
  }
  result *= (1+5.0*x) * (1+5.0*x/16.0) * (1+5.0*x/256);
  result /= (1-5.0*x/4.0) * (1-5.0*x/64.0) * (1-5.0*x/1024.0);
  return result;
}

long double Func_ini::operator()(long double x){
  return innerfn((1.0-x)/k-0.2) / x;
} 
\end{lstlisting}

\end{document}